\documentclass{article}

\usepackage{amsmath,amssymb,amsthm,color}
\usepackage[colorlinks=true]{hyperref}
\usepackage{pdfsync}

\newcommand{\N}{\mathbb{N}}
\newcommand{\R}{\mathbb{R}}
\renewcommand{\d}{\mathrm{d}}
\newcommand{\Vol}{\mathrm{Vol}}

\renewcommand{\leq}{\leqslant}
\renewcommand{\geq}{\geqslant}

\newtheorem{theorem}{Theorem}  
\newtheorem{proposition}{Proposition}
\newtheorem{corollary}{Corollary}
\newtheorem{definition}{Definition}
\newtheorem{lemma}{Lemma}
\theoremstyle{definition}\newtheorem{example}{Example}
\theoremstyle{definition}\newtheorem{remark}{Remark}

\title{Sub-Riemannian structures on groups of diffeomorphisms}

\date{}

\author{Sylvain Arguill\`ere\footnote{Johns Hopkins University, Center for imaging science, Baltimore, MD, USA ({\tt sarguil1@johnshopkins.edu}).}
\and
Emmanuel Tr\'elat\footnote{Sorbonne Universit\'es, UPMC Univ Paris 06, CNRS UMR 7598, Laboratoire Jacques-Louis Lions, and Institut Universitaire de France, F-75005, Paris, France (\texttt{emmanuel.trelat@upmc.fr}).}
}

\begin{document}

\maketitle

\begin{abstract}
In this paper, we define and study strong right-invariant sub-Riemannian structures on the group of diffeomorphisms of a manifold with bounded geometry. We derive the Hamiltonian geodesic equations for such structures, and we provide examples of normal and of abnormal geodesics in that infinite-dimensional context. The momentum formulation gives a sub-Riemannian version of the Euler-Arnol'd equation.
Finally, we establish some approximate and exact reachability properties for diffeomorphisms, and we give some consequences for Moser theorems.
\end{abstract}

\medskip


\medskip


\tableofcontents

\section{Introduction}
The purpose of this paper is to define and study right-invariant sub-Riemannian structures on the group of diffeomorphisms of a manifold with bounded geometry, with a particular emphasis on strong structures which are natural structures from a geometric viewpoint. Our work generalizes to the sub-Riemannian case former studies by Arnol'd (see \cite{A}) or by Ebin and Marsden (see\cite{EM}) done in the Riemannian case. In particular, we provide a suitable framework which paves the way towards addressing problems of fluid mechanics settled in a sub-Riemannian manifold.

Our work was initially motivated by problems arising in mathematical shape analysis (see\cite{YBOOK}). The general purpose of shape analysis is to compare several shapes while keeping track of their geometric properties. This is done by finding a deformation, mapping one shape onto the others, which minimizes a certain action that depends on the properties of the shape. 
Such approaches have been used in the analysis of anatomical organs in medical images (see \cite{GM}). A deformation can be viewed as the flow of diffeomorphisms generated by a time-dependent vector field (see \cite{DGM,T1,T2}). Indeed, when considering the studied shapes as embedded in a manifold $M$, diffeomorphisms induce deformations of the shape itself. Such deformations preserve local (such as the smoothness) and global (such as the number of self-intersections) geometric properties of the shape. The set of all possible deformations is then defined as the set of flows of time-dependent vector fields on the space $\mathcal{H}_e$ of ``infinitesimal transformations", which is a subspace of the Hilbert space $\Gamma^s(TM)$ of all vector fields on $M$ of Sobolev class $H^s$. We will show, in this paper, that the group of diffeomorphisms of $M$ inherits in such a way of a right-invariant sub-Riemannian structure.

\medskip

Recall that a sub-Riemannian manifold is a triple $(M,\mathcal{H},h)$, where $M$ is a (usually, finite-dimensional) manifold and $(\mathcal{H},h)$ is a Riemannian subbundle of the tangent space $TM$ of $M$, called \emph{horizontal distribution}, equipped with a Riemannian metric $h$. Horizontal curves are absolutely continuous curves on $M$ with velocity in $\mathcal{H}$. Their length is defined with respect to the metric $h$, and then the corresponding sub-Riemannian distance between two points is defined accordingly. We refer the reader to \cite{Bellaiche,MBOOK} for a survey on sub-Riemannian geometry in finite dimension.

For a Lie group $G$ with right Lie algebra $\mathfrak{g}$, a right-invariant sub-Riemannian structure is uniquely determined by the choice of a fixed subspace $\mathfrak{h}\subset\mathfrak{g}$ equipped with an inner product. The corresponding sub-Riemannian structure is then induced by right translations.

The group of diffeomorphisms $\mathcal{D}^s(M)$ of Sobolev class $H^s$ of a manifold $M$ with bounded geometry is a Hilbert manifold. It is a topological group for the composition of diffeomorphisms for which the right-composition by a fixed element is smooth. As a consequence, $\mathcal{D}^s(M)$ admits right-invariant vector fields, the space of which can be identified with the space $\Gamma^s(TM)$ of all vector fields on $M$ of Sobolev class $H^s$. This allows to define right-invariant sub-Riemannian structures on $\mathcal{D}^s(M)$, like for finite dimensional Lie groups, as follows: the choice of a pre-Hilbert space $(\mathcal{H}_e,\left\langle\cdot,\cdot\right\rangle)$, with $\mathcal{H}_e$ a subspace of $\Gamma^s(TM)$, generates by right translation a weak Riemannian subbundle $(\mathcal{H},\left\langle\cdot,\cdot\right\rangle)$ of $T\mathcal{D}^s(M)$. Horizontal curves $t\mapsto\varphi(t)$ of this structure are the flows of time-dependent vector fields $t\mapsto X(t)$ such that $X(t)\in \mathcal{H}_e$ almost everywhere. 

\medskip

There are two main difficulties emerging when studying such a structure. 

The first one is that the subbundle $\mathcal{H}$ obtained is, in general, only continuous. This is due to the fact that $\mathcal{D}^s(M)$ is \emph{not} a Lie group. One way to address that difficulty is to consider a subspace of vector fields of larger regularity, but as a counterpart the resulting subbundles are not closed. 

The second problem is that we are dealing here with an infinite-dimensional sub-Riemannian manifold.
To the best of our knowledge, sub-Riemannian geometry has been left very much unexplored in the infinite-dimensional context. A weak sub-Riemannian structure on $\mathcal{D}^\infty(M)$ has been briefly studied in \cite{KhesinLee} in relation to the transport of smooth volume forms by diffeomorphisms with a sub-Riemannian cost (also see \cite{FigalliRifford}). The metric in \cite{KhesinLee} is not right-invariant though.
A controllability theorem on smooth diffeomorphisms has been established in \cite{AC}, in an approaching context. We also mention the recent paper \cite{GMV}, in which the authors establish the geodesic equations for certain infinite-dimensional weak sub-Riemannian geometries. In this reference, the manifolds considered are modeled on general \emph{convenient spaces} (see \cite{KMBOOK}), and horizontal distributions are closed, with a closed complement. This framework does not involve the case of dense horizontal distributions, that we can deal with using the tools developped in the present paper (and which appear naturally, e.g., when studying shape spaces), and that are required in order to guarantee the smoothness of the structure.


In this paper, we define strong right-invariant sub-Riemannian structures on the group $\mathcal{D}^s(M)$ of diffeomorphisms, and we investigate some properties of such structures. We characterize the geodesics, which are critical points of the action $A$ between two end-points. In infinite dimension a serious difficulty emerges, due to the fact that the range of the differential of the end-point mapping needs not be closed. 
This problem is particularly significant in our setting, because the horizontal distribution itself is not closed in general. This leads to a new type of geodesics, which we call \emph{elusive geodesics}, that cannot be detected using an adequate version of the Pontryagin maximum principle (as will be extensively discussed).
However, we provide sufficient conditions for a curve to be a critical point of the action, in the form of certain Hamiltonian equations, allowing us to prove the existence of a global Hamiltonian geodesic flow. Finally, we establish some results generalizing to this infinite-dimensional setting the famous Chow-Rashevski theorem, and their applications to volume form transport.

\medskip

The paper is organized as follows.

In Section \ref{sec2}, we recall some results on manifolds with bounded geometry and on their groups of diffeomorphisms. We define strong right-invariant sub-Riemannian structures on these groups, and we establish the metric and geodesic completeness for the corresponding sub-Riemannian distance;
In Section \ref{sec_geod}, we compute the differential and the adjoint of the end-point mapping, and we establish the normal and abnormal Hamiltonian geodesic equations in the group of diffeomorphims. 
Examples of normal geodesics are provided in Section \ref{ex}, as well as a characterization of abnormal curves with Dirac momenta. Finally,
in Section \ref{sec3}, we change viewpoint to establish approximate and exact reachability results under appropriate sufficient conditions, generalizing the usual Chow-Rashevski and ball-box theorems of finite-dimensional sub-Riemannian geometry (see \cite{Bellaiche,MBOOK}).

\section{Sub-Riemannian structures on groups of diffeomorphisms}\label{sec2}
Let $d\in\N^*$, and let $(M,g)$ be a smooth oriented Riemannian manifold, where $M$ is a smooth manifold of dimension $d$ equipped with a Riemannian metric $g$.
We assume that $M$ has bounded geometry, that is, we assume that its global injectivity radius $\mathrm{inj}(M)$ is positive and that one of the following equivalent conditions is satisfied:
\begin{enumerate}
\item For every $i\in\N$, the Riemannian norm of the $i$-th covariant derivative of the curvature tensor of $M$ is bounded.
\item For every $i\in \N$, there exists $C_i>0$ such that $\vert \d^ig\vert\leq C_i$
in every normal coordinate chart (defined by the Riemannian exponential map of $M$) of radius $\mathrm{inj}(M)/2$ on $M$.
\end{enumerate}
In this section, we are going to define a sub-Riemannian structure on the manifold $\mathcal{D}^s(M)$, which is the connected component of the neutral element of the group of diffeomorphisms of $M$ of Sobolev class $H^s$.

\subsection{Definition of the manifold $\mathcal{D}^s(M)$}\label{sec_defi}
Let us first settle some notations and recall some results about the manifold structure of the space $H^s(M,N)$ of mappings from $M$ to $N$ of Sobolev class $H^s$, for $s$ large enough. Here, $(N,h)$ is another smooth Riemannian manifold with bounded geometry, and $s\in\N$.

Let $f:M\rightarrow N$ be a smooth mapping. Its differential is a section of the vector bundle
$T^*M\otimes_M f^*TN,$ where $f^*TN=\{(x,v)\in M\times TN \mid v\in T_{f(x)}N\}$.
The metrics $g$ and $h$ and the Levi-Civita connections on $M$ and $N$ induce (unique) Riemannian metrics and connections on $T^*M\otimes f^*TN$ (and bundles of tensors of higher order), which we use to define 
$
\vert \d f\vert_s^2=\sum_{i=0}^{s}\int_M\vert \tilde\nabla^if(x)\vert^2\, \d\Vol (x),
$ 
for every integer $s$, where $\Vol$ is the Riemannian volume on $M$, and $\tilde\nabla$ stands for the connection on $T^*M\otimes f^*TN$.
Then, for every integer $s>d/2$, we define the set $H^{s}(M,N)$ of functions from $M$ to $N$ of Sobolev class $H^s$ as that of $\mathcal{C}^1$ functions from $M$ to $N$ such that $\vert \d f\vert_{s-1}^2<+\infty$ (see \cite{ES,S}). 
The tangent space $T_f{H}^{s}(M,N)$ at $f$ is the set of measurable sections $X$ of $f^*TN$ such that
\begin{equation}\label{sobnorm}
\vert X\vert_s^2=\sum_{i=0}^{s}\int_M \vert \nabla^i X(x)\vert^2\, dx_g <+\infty,
\end{equation}
where $\nabla$ stands for the connection on $f^*TN$.
Note that any $X\in T_fH^s(M,N)$ is bounded and that there exists $C>0$ such that $h_{f(x)}(X(x),X(x))\leq C\vert X\vert_s$ for every $x\in N$.

The set $H^s(M,N)$ is a Hilbert manifold (for $s>d/2$), with tangent space at any $f\in H^s(M,N)$ given by $T_f H^s(M,N) = \{X\in H^s(M,TN) \mid \pi_N \circ X = f \}$, i.e., if $X_f \in T_f H^s(M,N)$ then $X_f(x)\in T_{f(x)} N$.

If $s>d/2+\ell$, with $\ell\in\N\setminus\{0\}$, then we have a continuous inclusion $H^s(M,N)\hookrightarrow \mathcal{C}^\ell(M,N)$.
Moreover, if $M$ is compact, then the inclusion $H^{s+1}(M,N)\hookrightarrow  H^{s'+1}(M,N)$, with $s'<s$, is compact.

Taking $M=N$, we are now in a position to define $\mathcal{D}^s(M)$ (see \cite{EM}). We denote by $H_0^s(M,M)$ the connected component of $e=\mathrm{id}_M$ in $H^s(M,M)$, and by $\mathrm{Diff}(M)$ the set of $\mathcal{C}^1$ diffeomorphims on $M$.

\begin{definition}
We define
$$
\mathcal{D}^s(M)=H_0^s(M,M)\cap \mathrm{Diff}(M), 
$$
that is, $\mathcal{D}^s(M)$ is the connected component of $e=\mathrm{id}_M$ in the space of diffeomorphisms of class $H^s$ on $M$.
\end{definition}

In the sequel, we denote by $\Gamma^s(TM)$ the set of vector fields of class $H^s$ on $M$.

\medskip

From now on, throughout the paper, we assume that $s> d/2+1$.

Then, the set $\mathcal{D}^s(M)$ is an open subset of $H^s(M,M)$, and thus is an Hilbert manifold. For every $\varphi\in \mathcal{D}^s(M)$, we have $T_\varphi \mathcal{D}^s(M) = T_\varphi H^s(M,M) = \Gamma^s(TM) \circ \varphi$.

The set $\mathcal{D}^s(M)$ is also a topological group for the composition $(\varphi,\psi)\mapsto\varphi\circ\psi$, and we list hereafter some of the regularity properties of the composition mappings (see \cite{ES,S}).

For every $\psi\in \mathcal{D}^s(M)$, the right multiplication $R_{\psi}:\varphi\mapsto\varphi\circ\psi$ is a smooth mapping from $\mathcal{D}^s(M)$ to $\mathcal{D}^s(M)$. 
The differential of $R_\psi$ at some point $\varphi\in \mathcal{D}^s(M)$ is the mapping $\d R_\psi: T_\varphi \mathcal{D}^s(M) \rightarrow T_{\varphi\circ\psi} \mathcal{D}^s(M)$ given by $\d R_\psi(\varphi).X=X\circ\psi$, for every $X\in T_\varphi \mathcal{D}^s(M)$.

For every $\varphi\in \mathcal{D}^s(M)$, the left multiplication $L_\varphi:\psi\mapsto\varphi\circ\psi$ is a continuous mapping from $\mathcal{D}^s(M)$ to $\mathcal{D}^s(M)$, but has no more regularity unless $\varphi$ itself is more regular. For instance, if $\varphi$ is of class $H^{s+1}$, then $L_\varphi$ is of class $\mathcal{C}^1$ and its differential at some point $\psi\in \mathcal{D}^s(M)$ is given by $\d L_\varphi(\psi)=(\d \varphi)\circ\psi$.
Actually, for every $k\in\N$, the mapping
$(\varphi,\psi)\in\mathcal{D}^{s+k}(M)\times \mathcal{D}^s(M)\mapsto\varphi\circ\psi\in
\mathcal{D}^s(M)$ is of class $\mathcal{C}^k$.

It is important to note that $\mathcal{D}^s(M)$ is not a Lie group because, although the right multiplication is smooth, the left composition in $\mathcal{D}^s(M)$ is only continuous. It can however be noticed that the set $\mathcal{D}^\infty(M)=\cap_{s> {d}/{2}+1}\mathcal{D}^s(M)$ is a Lie group. More precisely, endowed with the inverse limit topology, $\mathcal{D}^\infty(M)$ is an \emph{inverse limit Hilbert Lie group}, which is a particular type of a Fr\'echet Lie group (see \cite{O}).

\medskip

Let us now identify the right-invariant vector fields on $\mathcal{D}^s(M)$. Note first that the tangent space $T_e \mathcal{D}^s(M)$ of $\mathcal{D}^s(M)$ at $e=\mathrm{id}_M$ coincides with the space $\Gamma^s(TM)$ of vector fields of class $H^s$ on $M$. Using the smoothness of the right multiplication, we consider the set of right-invariant vector fields $\hat X$ on $\mathcal{D}^s(M)$, that is the set of vector fields $\hat X:\mathcal{D}^s(M)\rightarrow T\mathcal{D}^s(M)$ that satisfy $\hat X(\varphi)=\hat X(e)\circ\varphi$. 
Therefore, we have the identification:
$$
\{\textrm{right-invariant vector fields on}\ \mathcal{D}^s(M)\} \simeq T_e\mathcal{D}^s(M) \simeq \Gamma^s(TM).
$$
Although $\mathcal{D}^s(M)$ is not a Lie group, it behaves like a Lie group.

\paragraph{Curves on $\mathcal{D}^s(M)$.}
For every $\varphi(\cdot)\in H^1(0,1;\mathcal{D}^s(M))$, the time-dependent right-invariant vector field $X(\cdot) = \dot\varphi(\cdot)\circ\varphi(\cdot)^{-1} \in L^2(0,1;\Gamma^s(TM))$ is called the \emph{logarithmic velocity} of $\varphi(\cdot)$. Note that, by definition, we have $\dot\varphi(t) = X(t)\circ\varphi(t)$ for almost every $t\in[0,1]$.

Any curve $\varphi(\cdot)\in H^1(0,1;\mathcal{D}^s(M))$ of diffeomorphisms is the flow of a time-dependent right-invariant vector field on $\mathcal{D}^s(M)$ whose norm is square-integrable in time. Conversely, thanks to a generalized version of the Cauchy-Lipschitz theorem (see, e.g.,  \cite{TBOOK}), any time-dependent right-invariant vector field $X(\cdot)\in L^2(0,1;\Gamma^s(TM))$ generates a unique flow $\varphi^X_e(\cdot)\in H^1(0,1;\mathcal{D}^s(M))$ such that $\varphi^X_e(0)=e$, and therefore defines a unique curve $\varphi^X(\cdot)=\varphi_e^X(\cdot)\circ\varphi_0$ for any fixed $\varphi_0\in\mathcal{D}^s(M)$.

In other words, given any $\varphi_0\in\mathcal{D}^s(M)$, there is a one-to-one correspondence 
$$X(\cdot)\longleftrightarrow\varphi^X(\cdot)$$
between time-dependent vector fields $X(\cdot)\in L^2(0,1;\Gamma^s(TM))$ and curves $\varphi(\cdot)\in H^1(0,1;\mathcal{D}^s(M))$ such that $\varphi(0)=\varphi_0$.

\subsection{Sub-Riemannian structure on $\mathcal{D}^s(M)$}\label{sec_SRstructure}
A sub-Riemannian manifold is usually defined as a triple $(\mathcal{M},\Delta,h)$, where $\mathcal{M}$ is a manifold and $(\Delta,h)$ is a smooth Riemannian subbundle of the tangent space $T\mathcal{M}$ of $\mathcal{M}$, called \emph{horizontal distribution}, equipped with a Riemannian metric $h$. 

Here, we keep the framework and notations of Section \ref{sec_defi}. We take $\mathcal{M}=\mathcal{D}^s(M)$, and we are going to define a (right-invariant) horizontal distribution on the manifold $\mathcal{D}^s(M)$, endowed with a (right-invariant) Riemannian metric, that is, to define a (right-invariant) sub-Riemannian structure on the infinite-dimensional manifold $\mathcal{D}^s(M)$.

To this aim, in what follows we are going to consider a subspace $\mathcal{H}_e$ of $\Gamma^s(TM)$, consisting of vector fields that may have more regularity than $H^s$. This subset of vector fields is not necessarily closed. Indeed, in imaging problems the space $\mathcal{H}_e$ is often defined thanks to a heat kernel, as a RKHS (see Remark \ref{rkhs}), and then this set of analytic vector fields is a proper dense subset of $\Gamma^s(TM)$. We will provide examples hereafter.

\medskip

More precisely, let $k\in\N$ be arbitrary. Recall that we have assumed that $s>d/2+1$. Throughout the paper, we consider an arbitrary subset $\mathcal{H}_e\subset\Gamma^{s+k}(TM)$, endowed with a Hilbert product $\langle\cdot,\cdot\rangle$ such that $(\mathcal{H}_e,\langle\cdot,\cdot\rangle)$ has a continuous inclusion $\mathcal{H}_e\hookrightarrow \Gamma^{s+k}(TM)$.

\begin{definition}\label{srstruc}
We consider the subbundle $\mathcal{H}^s$ of $T\mathcal{D}^s(M)$ defined by $\mathcal{H}_\varphi^s=R_\varphi \mathcal{H}_e=\mathcal{H}_e\circ\varphi$ for every $\varphi\in\mathcal{D}^s(M)$, endowed with the (fibered) metric 
$\langle X,Y\rangle_\varphi=\langle X\circ\varphi^{-1},Y\circ\varphi^{-1}\rangle$. This subbundle induces a sub-Riemannian structure on $\mathcal{D}^s(M)$, that we refer to as the \emph{strong right-invariant sub-Riemannian structure} induced by $\mathcal{H}_e$ on $\mathcal{D}^{s}(M)$.
\end{definition}

Note that the mapping $(\varphi,X)\mapsto X\circ\varphi$ gives a parametrization of $\mathcal{H}^s$ by $\mathcal{D}^{s}(M)\times \mathcal{H}_e$, which is of class $\mathcal{C}^{k}$.

If $k\geq 1$, then, for any integer $s'$ such that $s<s'\leq{s}+k$, the restriction of $\mathcal{H}^s$ to $\mathcal{D}^{s'}(M)$ coincides with $\mathcal{H}^{s'}$. In particular, for a fixed diffeomorphism $\varphi\in \mathcal{D}^{s'}(M)$, we have $\mathcal{H}^{s}_\varphi=\mathcal{H}^{s'}_\varphi=\mathcal{H}_e\circ\varphi$. Hence $\mathcal{H}^{s}_\varphi$ does not depend on $s$ and we will simply write it as $\mathcal{H}_\varphi$.

\medskip

Let us provide hereafter two typical and important examples of a subset $\mathcal{H}_e$. 

\begin{example}\label{riem}
The simplest example of a strong right-invariant sub-Riemannian structure on the group of diffeomorphisms $\mathcal{D}^{s}(M)$ is obtained by taking $\mathcal{H}_e=\Gamma^{{s}+k}(TM)$ and $\langle X,X\rangle=\vert X\vert_{s+k}^2$ (as defined by \eqref{sobnorm}). Then $\mathcal{H}^s_{\varphi}$ is the set of all $X\in H^{s}(M,TM)$ such that $X(x)\in T_{\varphi(x)}M$ for every $x\in M$ and $X\circ\varphi^{-1}\in \Gamma^{s+k}(TM)$.

For $k=0$, we have $\mathcal{H}^s_{\varphi}=T\mathcal{D}^s(M)$ and we obtain a Riemannian structure on $\mathcal{D}^s(M)$. Moreover, using a careful computation and the change of variable $x=\varphi(y)$ in the integral of \eqref{sobnorm}, it can be seen that the Riemannian structure is actually smooth, not just continuous.

Such metrics have been studied in \cite{BBHM,BHM,MM}, seen as ``weak" Riemannian metrics on the group of smooth diffeomorphisms (in contrast with ``strong" metrics as in \cite{BV}).

For $k\geq 1$, $\mathcal{H}_e$ is a proper dense subset of $\Gamma^s(TM)$.
\end{example}


\begin{example}\label{subriem}
Let $\Delta$ be a smooth subbundle of $TM$, endowed with the restriction of the metric $g$ to $\Delta$. We define the space $\mathcal{H}_e=\left\lbrace X\in\Gamma^{s+k}(TM)\mid \forall x\in M\quad X(x)\in \Delta_x\right\rbrace$ of all horizontal vector fields of class $H^{s+k}$.
Then $\mathcal{H}_\varphi$ is the set of all $X\in H^s(M,TM)$ such that $X(x)\in \Delta_{\varphi(x)}$ for every $x\in M$ and $X\circ\varphi^{-1}\in \Gamma^{s+k}(TM)$.

If $k=0$ then $\mathcal{H}^s$ is a closed subbundle of $T\mathcal{D}^s(M)$, and hence it has a closed orthogonal supplement. We can note that this specific situation enters into the framework studied in \cite{GMV} where partial results concerning the geodesic equations have been established.

If $k\geq 1$ then $\mathcal{H}^s$ is neither closed, nor dense in $T\mathcal{D}^s(M)$.

Note that, although the parametrization $(X,\varphi)\mapsto X\circ\varphi$ is only continuous, $\mathcal{H}$ is actually a smooth subbundle of $T\mathcal{D}^{{s}}(M)$ because the subbundle $\Delta$ itself is smooth.

This example, where we design a sub-Riemannian structure on $\mathcal{D}^s(M)$ induced by a sub-Riemannian structure on the finite-dimensional manifold $M$, has also been considered with $s=+\infty$ in \cite{KhesinLee}, where Moser theorems for horizontal flows have been derived (see also Section \ref{sec_Moser} for such theorems), and in \cite{BGR1,BGR2} in order to handle corrupted data by means of hypoelliptic diffusions.
\end{example}

Note that, as we will see in Section \ref{sec3}, the case $k=0$ in Example \ref{subriem} is the only non-trivial case on which (up to our knowledge) we have an exact reachability result.

These examples show that the parametrization $(X,\varphi)\mapsto X\circ\varphi$ may not be the most appropriate one in some particular situations. We will however keep this point of view throughout the paper, so as to remain as general as possible.

\begin{remark}\label{euler}
In Definition \ref{srstruc}, we consider ``strong" sub-Riemannian structures. This is in contrast with ``weak" sub-Riemannian structures on $\mathcal{D}^s(M)$, for which the norm induced by $\langle\cdot,\cdot\rangle$ on $\mathcal{H}_e$ is not complete. Such weak structures allow to deal with more general metrics such as the one coming from the $L^2$ inner product on vector fields, given by $\langle X,Y\rangle=\int_Mg_x(X(x),Y(x))\,\d\Vol_g(x)$. When $\mathcal{H}_e$ is the set of volume-preserving vector fields, this metric induces a smooth Riemannian metric on the group of volume-preserving diffeomorphisms of $M$, whose geodesics are the solutions of the Euler equations on $M$ (see \cite{A,EM}).
\end{remark}

\begin{remark}\label{rkhs}
Any Hilbert space $\mathcal{H}_e$ of vector fields of class at least $H^{s}$, with $s> d/2$, admits a \emph{reproducing kernel} \cite{TY2,YBOOK}. This means that the operator $\mathcal{H}_e^*\rightarrow \mathcal{H}_e$ given by the inverse of the isometry $X\mapsto \langle X,\cdot\rangle$, is the convolution with a section $K$ of the bundle $L(T^*M,TM)=TM\otimes TM\rightarrow M\times M$, called the reproducing kernel of $\mathcal{H}_e$.

Any element $P$ of the dual $\Gamma^{-s}(T^*M)$ of $\Gamma^s(TM)$ can be represented by a one-form with (distributional) coefficients of class $H^{-s}$, so that $P(X)=\int_M P(x)(X(x))\,\d x$, for every $X\in \Gamma^s(TM)$ (the integral is computed in coordinates by means of a partition of unity). Such a $P$ is called a \emph{co-current}.
By restriction, any co-current $P\in\Gamma^{-s}(T^*M)$ also belongs to the dual $\mathcal{H}_e^*$ of $\mathcal{H}_e$. In particular, for any $(x,p)\in T^*M$, the linear form $p\otimes \delta_x:X\mapsto p(X(x))$ belongs to $\mathcal{H}_e^*$, and the unique element $Y$ such that $\langle Y,\cdot\rangle=p\otimes\delta_x$ on $\mathcal{H}_e$ is denoted by $K(\cdot,x)p$. 
Then $K(x,y)$ is a linear mapping from $T^*_yM$ to $T_xM$, that is an element of $T_yM\otimes T_xM$. Moreover, for every co-current $P\in \Gamma^{-s}(T^*M)$, the unique element $Y$ such that $\langle Y,\cdot\rangle=P$ on $\mathcal{H}_e$ is given by $Y(x)=\int_{M}K(x,y)P(y)\, \d\Vol_g(y)$.
In particular, $K$ is of class at least  $H^s$ on $M\times M$.
\end{remark}

We are now in a position to consider horizontal curves on $\mathcal{D}^s(M)$ for the strong right-invariant sub-Riemannian structure induced by $\mathcal{H}_e$ (see Definition \ref{srstruc}), and to define the corresponding concept of sub-Riemannian distance.

\subsection{Horizontal curves and end-point mapping}
Recall that $\mathcal{H}_e$ has a continuous inclusion in $\Gamma^{s+k}(TM)$, with $k\geq0$ and $s>d/2+1$ integers.

\begin{definition}
An \emph{horizontal curve} for the strong right-invariant sub-Riemannian structure induced by $(\mathcal{H}_e,\langle\cdot,\cdot\rangle)$ on $\mathcal{D}^s(M)$ is a curve $\varphi(\cdot)\in H^1(0,1;\mathcal{D}^s(M))$ such that
$\dot\varphi(t)\in\mathcal{H}^s_{\varphi(t)}$ for almost every $t\in[0,1]$.
Equivalently, $\varphi(\cdot)$ is the right-translation of the flow of a time-dependent vector field $X(\cdot)\in L^2(0,1;\mathcal{H}_e)$.
\end{definition}

For every $\varphi_0\in \mathcal{D}^s(M)$, we define $\Omega_{\varphi_0}$ as the set of all horizontal curves $\varphi(\cdot)\in H^1(0,1;\mathcal{D}^s(M))$ such that $\varphi(0)=\varphi_0$.
We define the mapping $\Phi_{\varphi_0}:L^2(0,1;\mathcal{H}_e)\rightarrow\Omega_{\varphi_0}$ by $\Phi_{\varphi_0}(X(\cdot)) = \varphi^X(\cdot)$, where $\varphi^X(\cdot)$ is the unique solution of the Cauchy problem $\dot\varphi^X(\cdot)=X(\cdot)\circ\varphi^X(\cdot)$, $\varphi^X(0)=\varphi_0$. We have $\Omega_{\varphi_0}=\Phi_{\varphi_0}(L^2(0,1;\mathcal{H}_e))$.

\begin{lemma}\label{horcur}
We assume that $k\geq 1$.
For every $\varphi_0\in \mathcal{D}^s(M)$,  the mapping $\Phi_{\varphi_0}:L^2(0,1;\mathcal{H}_e)\rightarrow\Omega_{\varphi_0}$ is a $\mathcal{C}^k$ diffeomorphism, and the set $\Omega_{\varphi_0}$ is a $\mathcal{C}^k$ submanifold of $H^1(0,1;\mathcal{D}^s(M))$.
\end{lemma}

\begin{proof}
Using the correspondence $\varphi(\cdot)\leftrightarrow X(\cdot)$ described in Section \ref{sec_defi} (for $\varphi_0$ fixed), it suffices to prove that the graph in $H^1(0,1;\mathcal{D}^s(M))\times L^2(0,1;\mathcal{H}_e)$ of the $\mathcal{C}^k$ mapping $X(\cdot)\mapsto \varphi^X(\cdot)$ is a $\mathcal{C}^k$ manifold, globally parametrized by $X(\cdot)\mapsto (\varphi^X(\cdot),X(\cdot)\circ\varphi^X(\cdot))$.

We denote by $H_{\varphi_0}^1(0,1;\mathcal{D}^s(M))$ the set of $\varphi(\cdot)\in H^1(0,1;\mathcal{D}^s(M))$ such that $\varphi(0)=\varphi_0$.
We define $H^1_{\varphi_0}\times L^2(0,1;T\mathcal{D}^s(M))$ as the fiber bundle over $H^1_{\varphi_0}(0,1;T\mathcal{D}^s(M))$ defined by
$\{ (\varphi(\cdot),\delta\varphi(\cdot))\in H_{\varphi_0}^1(0,1;\mathcal{D}^s(M))\times L^2(0,1;T\mathcal{D}^s(M)) \mid  \delta\varphi(t)\in T_{\varphi(t)}\mathcal{D}^s(M)\ \text{for a.e.}\ t\in[0,1]  \}. $
We consider the affine vector bundle morphism $C: H^1(0,1;\mathcal{D}^s(M))\times L^2(0,1;\mathcal{H}_e)\rightarrow H^1_{\varphi_0}\times L^2(0,1;T\mathcal{D}^s(M))$
defined by $C(\varphi(\cdot),X(\cdot))(t)=\dot{\varphi}(t)-X(t)\circ\varphi(t)$.
Then $C$ is of class $\mathcal{C}^k$ and $\Omega_{\varphi_0}=C^{-1}(\{0\})$. In coordinates on $\mathcal{D}^s(M)$, we have $\partial_\varphi C(\varphi(\cdot),X(\cdot)).\delta\varphi(\cdot)=\delta \dot\varphi(\cdot)-\d (X(\cdot)\circ\varphi(\cdot)).\delta\varphi(\cdot)$, which is a continuous linear differential operator of order one, and the linear Cauchy-Lipschitz theorem in Banach spaces  implies that it is an isomorphism. The lemma follows from the implicit function theorem.
\end{proof}

We will often identify a horizontal curve $\varphi(\cdot)\in \Omega_{\varphi_0}$ with its logarithmic velocity $X(\cdot)=\dot{\varphi}(\cdot)\circ\varphi(\cdot)^{-1}$.


\begin{definition}\label{end-point}
For every $\varphi_0\in\mathcal{D}^s(M)$, the \emph{end-point mapping} $\mathrm{end}_{\varphi_0}:\Omega_{\varphi_0}\rightarrow \mathcal{D}^s(M)$ is defined by $\mathrm{end}_{\varphi_0}(\varphi(\cdot))=\varphi(1)$, for every $\varphi(\cdot)\in \Omega_{\varphi_0}$.
\end{definition}

The end-point mapping is of class $\mathcal{C}^k$. 

\begin{definition}\label{def_singular}
An horizontal curve $\varphi(\cdot)\in H^1(0,1;\mathcal{D}^s(M))$ is said to be \emph{singular} if the codimension of $\mathrm{Range}(\d\,\mathrm{end}_{\varphi_0}(\varphi(\cdot)))$ in $T_{\varphi_1}\mathcal{D}^s(M)$ is positive.
\end{definition}

\begin{remark}\label{rem_sing2}
An horizontal curve $\varphi(\cdot)\in H^1(0,1;\mathcal{D}^s(M))$ is singular if and only if there exists $P_{\varphi_1}\in T^*_{\varphi_1}\mathcal{D}^s(M)\setminus\{0\}$ such that $(\d\,\mathrm{end}_{\varphi_0}(\varphi(\cdot)))^*.P_{\varphi_1}=0$. 

Examples of singular curves of diffeomorphisms are provided in Section \ref{singdirac}. We will see that such curves can easily be built by considering a sub-Riemannian manifold $M$ on which there exists a nontrivial singular curve $\gamma(\cdot)$ (such manifolds do exist, see, e.g., \cite{MBOOK}), and by taking $\mathcal{H}_e$ as the set of horizontal vector fields of class $H^{{s}+k}$ on $M$, as explained in Example \ref{subriem}. Then the flow of any horizontal vector field $X$ such that $X\circ\gamma(\cdot)=\dot{\gamma}(\cdot)$ happens to be a singular curve of diffeomorphisms.

Theorem \ref{thm_abnormal} (further) will provide an Hamiltonian characterization of singular curves.
\end{remark}

Given $\varphi_0$ and $\varphi_1$ in $\mathcal{D}^s(M)$, we consider the set $\Omega_{\varphi_0,\varphi_1}=\mathrm{end}_{\varphi_0}^{-1}(\{\varphi_1\})$ of horizontal curves steering $\varphi_0$ to $\varphi_1$. 

\begin{remark}\label{rem_sing}
The set $\Omega_{\varphi_0,\varphi_1}$ need not be a submanifold of $\Omega_{\varphi_0}$, due to the fact that $\mathrm{end}_{\varphi_0}$ need not be a submersion.
In the finite-dimensional context, a singularity of this set is exactly a singular curve, that is a critical point of the end-point mapping, or equivalently, the projection of an abnormal extremal (see \cite{ChitourJeanTrelatJDG,MBOOK}).
In infinite dimension, the situation is more complicated because of the possible existence of proper subsets that are dense.
More precisely, since we are in infinite dimension, either of the three following possibilities may occur:
\begin{enumerate}
\item $\mathrm{Range}(\d\,\mathrm{end}_{\varphi_0}(\varphi(\cdot))) = T_{\varphi_1}\mathcal{D}^s(M)$;
\item the codimension of $\mathrm{Range}(\d\,\mathrm{end}_{\varphi_0}(\varphi(\cdot)))$ in $T_{\varphi_1}\mathcal{D}^s(M)$ is positive;
\item $\mathrm{Range}(\d\,\mathrm{end}_{\varphi_0}(\varphi(\cdot)))$ is a proper dense subset of $T_{\varphi_1}\mathcal{D}^s(M)$.
\end{enumerate}
In finite dimension, only the first two possibilities occur. In the first case it is usually said that $\varphi(\cdot)$ is \emph{regular}, and in that case $\Omega_{\varphi_0,\varphi_1}$ is, locally around $\varphi(\cdot)$, a submanifold of $\Omega_{\varphi_0}$.

However, in the present infinite-dimensional framework, the first possibility never occurs. Indeed, it is required that $k\geq1$ for the end-point mapping to have a differential, and then we have
$$
\mathrm{Range}(\d\,\mathrm{end}_{\varphi_0}(\varphi(\cdot)))\subset \Gamma^{s+1}(TM)\circ \varphi_1\varsubsetneq T_{\varphi_1}\mathcal{D}^s(M).
$$ 
In particular, we have to deal with the third possibility.
In the context of controlled partial differential equations, this third possibility corresponds to a situation where the control system is approximately controllable but not exactly controllable.
\end{remark}

\subsection{Sub-Riemannian distance and action}

\begin{definition}\label{defaction}
The \emph{sub-Riemannian length} $L(\varphi(\cdot))$ and the \emph{sub-Riemannian action} $A(\varphi(\cdot))$ of an horizontal curve $\varphi(\cdot)\in H^1(0,1;\mathcal{D}^s(M))$ with logarithmic velocity $X(\cdot)=\dot\varphi(\cdot)\circ\varphi(\cdot)^{-1} \in L^2(0,1;\mathcal{H}_e)$ are respectively defined by
$$
L(\varphi(\cdot)) = \int_0^1 \sqrt{\langle X(t),X(t)\rangle}\, \d t \quad \textrm{and}\quad 
A(\varphi(\cdot))= \frac{1}{2}\int_0^1 {\langle X(t),X(t)\rangle}\, \d t .
$$
The \emph{sub-Riemannian distance} $d_{SR}$ between two elements $\varphi_0$ and $\varphi_1$ of $\mathcal{D}^s(M)$ is defined as the infimum of the length of horizontal curves steering $\varphi_0$ to $\varphi_1$, with the agreement that $d_{SR}(\varphi_0,\varphi_1)=+\infty$ whenever there is no horizontal curve steering $\varphi_0$ to $\varphi_1$.

An horizontal curve $\varphi(\cdot):[0,1]\rightarrow \mathcal{D}^s(M)$ is said to be \emph{minimizing} if $d_{SR}(\varphi(0),\varphi(1)) = L(\varphi(\cdot))$. 
\end{definition}

We have $d_{SR}(\varphi_0,\varphi_1)=d_{SR}(e,\varphi_1\circ\varphi_0^{-1})$, that is, $d_{SR}$ is right-invariant.
Concatenation and time-reversals of horizontal curves are horizontal as well, hence $d_{SR}$ is at least a semi-distance, and the subset $\{\varphi_1\in \mathcal{D}^s(M)\mid  d_{SR}(e,\varphi_1)<\infty\}$ is a subgroup of $\mathcal{D}^s(M)$ that does not depend on $s$ (it is proved in Theorem \ref{complete} below that $d_{SR}$ is a distance). Moreover, it follows from the Cauchy-Schwarz inequality that $L^2\leq 2A$, and therefore $d_{SR}(\varphi_0,\varphi_1)$ is equal to $\sqrt{2}$ times the infimum of the action over all horizontal curves steering $\varphi_0$ to $\varphi_1$. Therefore, as in classical finite-dimensional sub-Riemannian geometry, minimizing the length over horizontal curves between two end-points is equivalent to minimizing the action over this set.

\begin{theorem}\label{complete}
The sub-Riemannian distance $d_{SR}$ is indeed a distance (taking its values in $[0,+\infty]$), that is, $d_{SR}(\varphi_0,\varphi_1)=0$ implies $\varphi_0=\varphi_1$. Moreover, any two elements $\varphi_0$ and $\varphi_1$ of $\mathcal{D}^s(M)$ with $d_{SR}(\varphi_0,\varphi_1)<+\infty$ can be connected by a minimizing horizontal curve, and $(\mathcal{D}^s(M),d_{SR})$ is a complete metric space.
\end{theorem}

\begin{proof}
The proof follows, in the context of sub-Riemannian geometry, some arguments of \cite{T1}.

Firstly, in order to prove that $d_{SR}(\varphi_0,\varphi_1)=0$ implies $\varphi_0=\varphi_1$, since $d_{SR}$ is right-invariant, it suffices to prove that $\varphi_1\neq e$ implies $d_{SR}(e,\varphi_1)>0$.
Let $\varphi(\cdot) \in\Omega_{e}$ be an horizontal curve such that $\varphi(0)=e$ and $\varphi(1)=\varphi_1\neq e$, and let $X(\cdot)$ be its logarithmic velocity. Let $x\in M$ such that $\varphi_1(x)\neq x$. Setting $x(t)=\varphi(t,x)$, we have $\dot{x}(t)=X(t,x(t))$, and then, $0<d_M(x,\varphi_1(x))^2\leq \int_0^1g_{x(t)}(X(t,x(t)),X(t,x(t)))\,\d t
$, where $d_M$ is the Riemannian distance on $M$. Since there exist positive constants $C_1$ and $C_2$ such that $g_y(Y(y),Y(y))\leq C_1\vert Y\vert^2_s\leq C_2\langle Y,Y\rangle$, for every $y\in M$ and for every $Y\in \mathcal{H}_e$, it follows that
$0<d_M(x,\varphi_1(x))^2\leq C_2\int_0^1\langle X(t),X(t)\rangle\,\d t=2C_2A(\varphi(\cdot))$.
Since this inequality is true for every horizontal curve $\varphi(\cdot)$ steering $e$ to $\varphi_1$, we get that $0<d_M(x,\varphi_1(x))\leq d_{SR}(e,\varphi_1)$ (note by the way that this inequality holds true for every $x\in M$).

Secondly, since the structure is right-invariant, it suffices to prove that any $\varphi_1\in \mathcal{D}^s(M)$ such that $d_{SR}(e,\varphi_1)<\infty$ can be reached from $e$ by a minimizing horizontal curve. In order to prove this fact, we use the following lemma, which itself mainly follows from the Sobolev embedding theorem (see \cite{ATY} for the proof). 

\begin{lemma}\label{seqcomp}
Let $(X^n(\cdot))_{n\in\N}$ be a bounded sequence of $L^2(0,1;\mathcal{H}_e)$, consisting of logarithmic velocities of horizontal curves $(\varphi^n(\cdot))_{n\in\N}$ such that $\varphi^n(0)=e=\mathrm{id}_M$. Then, there exist $\bar X(\cdot)\in L^2(0,1;\mathcal{H}_e)$, corresponding to the horizontal curve $\bar\varphi(\cdot)$, and an increasing sequence $({n_j})_{j\in\N}$ of integers such that $(X^{n_j}(\cdot))_{j\in\N}$ converges weakly to $\bar X(\cdot)$ as $j$ tends to $+\infty$, and such that
$$
\sup_{t\in[0,1]}d_{H^{{s}+k-1}(U,M)}(\varphi^{n_j}(t),\bar\varphi(t)) \rightarrow 0 ,
$$ 
as $j$ tends to $+\infty$, for every compact subset $U$ of $M$.
\end{lemma}

Let $(X^n(\cdot))_{n\in\N}$ be a minimizing sequence of $L^2(0,1;\mathcal{H}_e)$, associated with horizontal curves $(\varphi^n(\cdot))_{n\in\N}$, for the problem of minimizing the action over all horizontal curves steering $e$ to $\varphi_1$. By Lemma \ref{seqcomp}, up to some subsequence, the sequence $(X^n(\cdot))_{n\in\N}$ converges weakly to $\bar{X}(\cdot)$, associated with an horizontal curve $\bar\varphi(\cdot)$ such that $\bar\varphi(1)=\varphi_1$ (note that $s+k-1> d/2$). Hence $\bar\varphi(\cdot)$ is an horizontal curve steering $e$ to $\varphi_1$, and by lower semi-continuity of the action, we have $A(\bar\varphi(\cdot))=\frac{1}{2}\int_0^1\langle\bar{X}(t),\bar{X}(t)\rangle\,\d t\leq\liminf_n\frac{1}{2}\int_0^1\langle X^n(t),X^n(t)\rangle\,\d t$, and hence $\bar\varphi(\cdot)$ is a minimizing horizontal curve steering $e$ to $\varphi_1$.

Let us finally prove that $(\mathcal{D}^s(M),d_{SR})$ is complete. Let $(\varphi_1^n)_{n\in \N}$ be a Cauchy sequence of $\mathcal{D}^s(M)$. Then $(\varphi_1^n)_{n\in\N}$ is a Cauchy sequence as well in $\mathcal{C}^0(M,M)$ for the metric topology of uniform convergence on compact subsets, which is complete, and therefore converges to some $\bar\varphi_1\in \mathcal{C}^0(M,M)$. 
To conclude, it suffices to prove that $\bar\varphi_1\in\mathcal{D}^s(M)$ and that $d_{SR}(\varphi_1^n,\bar\varphi_1)\rightarrow 0$ as $n\rightarrow+\infty$.

For every integer $m\geq n$, let $X^{n,m}(\cdot)$ be the logarithmic velocity of a minimizing horizontal curve $\varphi^{n,m}(\cdot)$ steering $\varphi_1^n$ to $\varphi_1^m$. For every $n$, the sequence $(X^{n,m}(\cdot))_{m\geq n}$ (indexed by $m$) is bounded in $L^2(0,1;\mathcal{H}_e)$, and hence from Lemma \ref{seqcomp}, up to some subsequence it converges weakly to some $\bar X^n(\cdot)\in L^2(0,1;\mathcal{H}_e)$, which is the logarithmic velocity of an horizontal curve $\bar\varphi^n(\cdot)$. Moreover $\varphi_1^m=\varphi^{n,m}(1)\rightarrow\bar\varphi^n(1)$ in $\mathcal{C}^0(M,M)$ as $m\rightarrow+\infty$. But since $\varphi_1^m\rightarrow\bar\varphi_1$, it follows that $\bar\varphi^n(1)=\bar\varphi_1$. In particular, $\bar\varphi_1 \in \mathcal{D}^s(M)$ and $\bar\varphi_1^n(\cdot)$ is an horizontal curve steering $\varphi_1^n$ to $\bar\varphi_1$. 
By weak convergence of $X^{n,m}(\cdot)$ to $\bar X^n(\cdot)$ as $m\rightarrow+\infty$, and by lower semi-continuity, we infer that
$$
\begin{aligned}
d_{SR}(\bar\varphi_1,\varphi_1^n)^2 &\leq 2A(\bar\varphi^n(\cdot)) = \int_0^1\langle\bar X^n(t),\bar X^n(t)\rangle\,\d t\\
&\leq
\liminf_{m\rightarrow\infty}\int_0^1\langle X^{n,m(t)}, X^{n,{m}}(t)\rangle\,\d t
=
\liminf_{m\rightarrow\infty} 2A(\varphi^{n,m}(\cdot)) = \liminf_{k\rightarrow\infty} d_{SR}(\varphi_1^n,\varphi_1^m)^2.
\end{aligned}
$$
The equality in the last part is due to the fact that $\varphi^{n,m}(\cdot)$ is a minimizing horizontal curve steering $\varphi_1^n$ to $\varphi_1^m$. 
Since $(\varphi_1^n)_{n\in\N}$ is a Cauchy sequence, the right-hand side of the above inequality tends to $0$ as $n\rightarrow+\infty$, and hence $d_{SR}(\bar\varphi_1,\varphi_1^n)\rightarrow 0$ as $n\rightarrow+\infty$.
\end{proof}

\begin{remark}
The topology defined by $d_{SR}$ is always finer or as coarse as the manifold topology of $\mathcal{D}^s(M)$. Indeed, $ C'\vert X\vert_s\leq\left<X,X\right>$ implies that $d_{SR}$ is greater than $C'$ multiplied by the strong Riemannian distance induced by the structure described in Example \ref{riem} with $k=0$. But it was proved in \cite{BV} that this last metric induces the intrinsic manifold topology on $\mathcal{D}^s(M)$.
\end{remark}

\section{Geodesics on $\mathcal{D}^s(M)$}\label{sec_geod}
We keep the framework and notations used in the previous sections. 

\begin{definition}\label{def_geod}
A \emph{geodesic} $\varphi(\cdot)\in H^1(0,1;\mathcal{D}^s(M))$ is an horizontal curve which is a critical point of the action mapping $A$ restricted to $\Omega_{\varphi(0),\varphi(1)}$. 
In other words, for any $\mathcal{C}^1$ family of horizontal curves $s\in (-\varepsilon,\varepsilon)\mapsto\varphi^s(\cdot)\in \Omega_{\varphi(0),\varphi(1)}$, with $\varepsilon>0$ and $\varphi^0(\cdot)=\varphi(\cdot)$, we have $\d A(\varphi(\cdot)).\partial_s\varphi^s(\cdot)_{\vert s=0}=0$. With a slight abuse of notation, we will denote by $T_{\varphi(\cdot)}\Omega_{\varphi(0),\varphi(1)}$ the set of all such infinitesimal variations $\partial_s\varphi^s(\cdot)_{\vert s=0}$.

A geodesic $\varphi(\cdot)$ is said to be \emph{minimizing} if $L(\varphi(\cdot))=d_{SR}(\varphi(0),\varphi(1))$.
\end{definition}

Note that, obviously, any minimizing horizontal curve is a geodesic.

\subsection{Preliminary discussion: Lagrange multipliers}
In finite dimension, the critical point property usually leads to a Lagrange multipliers relation, which provides a first-order necessary condition for optimality, itself allowing us to derive Hamiltonian geodesic equations. Here, since we are in infinite dimension, the situation is more complex and we do not have necessarily a nontrivial Lagrange multiplier.
Let us be more precise with this important difficulty, because it justifies the point of view that we are going to adopt in the sequel. The discussion goes as follows.

Let $\varphi(\cdot)\in H^1(0,1;\mathcal{D}^s(M))$ be a minimizing horizontal curve steering $\varphi_0\in\mathcal{D}^s(M)$ to $\varphi_1\in\mathcal{D}^s(M)$. Then $\varphi(\cdot)$ is a geodesic, that is, a critical point of the action $A$ restricted to $\Omega_{\varphi_0,\varphi_1}=\mathrm{end}_{\varphi_0}^{-1}(\{\varphi_1\})$. Defining the mapping $E_{\varphi_0} : L^2(0,1;\mathcal{H}_e)\rightarrow \mathcal{D}^s(M)$ by the composition $E_{\varphi_0} = \mathrm{end}_{\varphi_0}\circ\Phi_{\varphi_0}$, and defining the mapping $F_{\varphi_0} : L^2(0,1;\mathcal{H}_e)\rightarrow \mathcal{D}^s(M)\times \R$ by
$$
F_{\varphi_0}(X(\cdot)) = ( E_{\varphi_0}(X(\cdot)) , A(\Phi_{\varphi_0}(X(\cdot))) ) = ( \mathrm{end}_{\varphi_0}(\varphi^X(\cdot)) , A(\varphi^X(\cdot)) ) ,
$$
it follows that the logarithmic velocity $X(\cdot)$ of $\varphi(\cdot)$ is a critical point of $F_{\varphi_0}$; or, in other words, the differential $\d F_{\varphi_0}(X(\cdot))$ is not surjective, that is, $\mathrm{Range}(\d F_{\varphi_0}(X(\cdot)))$ is a proper subset of $T_{\varphi_1}\mathcal{D}^s(M)\times \R$.
Then, there are two possible cases:
\begin{enumerate}
\item either the codimension of $\mathrm{Range}(\d F_{\varphi_0}(X(\cdot)))$ in $T_{\varphi_1}\mathcal{D}^s(M)\times \R$ is positive, which is equivalent to the fact that $\ker ((\d F_{\varphi_0}(X(\cdot)))^*)\neq\{0\}$,
\item or the space $\mathrm{Range}(\d F_{\varphi_0}(X(\cdot)))$ is dense in $T_{\varphi_1}\mathcal{D}^s(M)\times \R$, which is equivalent to the fact that $\ker ((\d F_{\varphi_0}(X(\cdot)))^*)=\{0\}$.
\end{enumerate}
The first case means that we have a nontrivial Lagrange multiplier, and the second case means that there does not exist any nontrivial Lagrange multiplier. Note that the second case can never occur in finite dimension. Here, since we are in infinite dimension, we have to face with this additional difficulty.

Since $\d\Phi_{\varphi_0} : L^2(0,1;\mathcal{H}_e) \rightarrow T_{\varphi(\cdot)}\Omega_{\varphi_0}$ is an isomorphism, it follows that, for the geodesic $\varphi(\cdot)=\varphi^X(\cdot)$, there are two possible issues:
\begin{enumerate}
\item There exists $(P_{\varphi_1},p^0)\in T^*_{\varphi_1}\mathcal{D}^s(M)\times\R\setminus\{(0,0)\}$ such that 
\begin{equation}\label{Lag_mult}
(\d\,\mathrm{end}_{\varphi_0}(\varphi(\cdot)))^*.P_{\varphi_1} + p^0 \d A(\varphi(\cdot)) =0.
\end{equation}
This is a Lagrange multipliers relation.
In finite dimension, only this first possibility does occur, and leads to the Pontryagin maximum principle (see, e.g., \cite{TBOOK} for this point of view). Note that the Lagrange multiplier $(P_{\varphi_1},p^0)$ is defined up to some multiplying scalar, and usually it is normalized by distinguishing between two subcases:
\begin{enumerate}
\item \emph{Normal case:} $p^0\neq 0$. In that case, we normalize the Lagrange multiplier so that $p^0=-1$. Then \eqref{Lag_mult} implies that $\d A(\varphi(\cdot)) = (\d\,\mathrm{end}_{\varphi_0}(\varphi(\cdot)))^*.P_{\varphi_1}$, and in that case we will then derive the so-called \emph{normal geodesic equations}.
\item \emph{Abnormal case:} $p^0=0$. In that case, \eqref{Lag_mult} implies that $(\d\,\mathrm{end}_{\varphi_0}(\varphi(\cdot)))^*.P_{\varphi_1}=0$ (and for instance we can normalize the Lagrange multiplier by normalizing $P_{\varphi_1}$). This is equivalent to saying that the corank of $\d\,\mathrm{end}_{\varphi_0}(\varphi(\cdot))$ is positive: in other words, according to Definition \ref{def_singular}, $\varphi(\cdot)$ is a singular curve. In that case, we will then derive the co-called \emph{abnormal geodesic equations}, which are the Hamiltonian characterization of  singular curves.

\begin{remark}
If $\varphi(\cdot)$ is a singular curve, then there exists $P_{\varphi_1}\in T^*_{\varphi_1}\mathcal{D}^s(M)\setminus\{0\}$ such that $(\d\,\mathrm{end}_{\varphi_0}(\varphi(\cdot)))^*.P_{\varphi_1}=0$. In other words, there exists an abnormal Lagrange multiplier. This Lagrange multiplier is not necessarily unique (up to some multiplying scalar), and the dimension of the space of such Lagrange multipliers is usually called the corank of the singular curve (see \cite{ChitourJeanTrelatJDG,ChitourJeanTrelatSICON} where generic properties are established for singular curves in finite dimension).
\end{remark}

\end{enumerate}

\item The mapping $(\d F_{\varphi_0}(X(\cdot)))^*$ is injective, that is, if we have \eqref{Lag_mult} for some $(P_{\varphi_1},p^0)\in T^*_{\varphi_1}\mathcal{D}^s(M)$, then $(P_{\varphi_1},p^0)=(0,0)$.

This case is peculiar to the infinite-dimensional setting, and can never occur in finite dimension.
In that case a necessary condition for optimality in the form of a Pontryagin maximum principle cannot be derived (see \cite[Chapter 4]{LiYong}). This leads us to state the following definition.
\end{enumerate}

\begin{definition}
A geodesic $\varphi(\cdot)\in H^1(0,1;\mathcal{D}^s(M))$ is said to be \emph{elusive} whenever the mapping $(\d F_{\varphi_0}(X(\cdot)))^*$ is injective.
\end{definition}

In what follows, we are going to state a Pontryagin maximum principle for non-elusive geodesics (that is, for the first case of the above discussion), and derive the normal and abnormal geodesic equations.

\begin{remark}\label{rem_elusive}
The concept of elusive geodesic is new, and is specific to the infinite dimension. What is important to understand is that elusive geodesics escape to the dual analysis in terms of Lagrange multipliers, due to the topology of the ambient space.

As it follows from the definition of an elusive curve, a Lagrange multiplier cannot exist because, although the mapping $\d F_{\varphi_0}(X(\cdot))$ is not surjective, its range is however dense in the target space $T_{\varphi_1}\mathcal{D}^s(M)\times\R$. This difficulty, which is specific to the infinite-dimensional setting, is actually well known in constrained optimization. In \cite{Kurcyusz} the author provided some weak regularity conditions under which the existence of Lagrange multipliers can be established for a general nonlinear programming problem. He showed that the topology of the target space plays an important role, and he established a connection between the choice of suitable function spaces and the existence of Lagrange multipliers.

\smallskip

Before commenting on this choice, let us first provide an easy way to exhibit elusive geodesics. The idea relies on the fact that $\mathcal{H}_e$ is not closed in $\Gamma^s(TM)=T_e\mathcal{D}^s(M)$, which results in ``missing" some initial momenta. For example, if $\varphi_0=e$ and if $k\geq 3$, then restricting $\mathcal{H}^s$ to $\mathcal{D}^{s+1}(M)$ gives new initial momenta $P_0$, namely, those belonging to $\Gamma^{-s-1}(T^*M)\setminus\Gamma^{-s}(T^*M)$ (where $\Gamma^{-s}(T^*M)$ is the dual of $\Gamma^s(TM)=T_e\mathcal{D}^s(M)$).

Based on this idea, the method to exhibit elusive geodesics consists of ``decreasing the order" of the cotangent space, in the following sense. The Hilbert space $(\mathcal{H}_e,\left\langle\cdot,\cdot\right\rangle)$, with $\mathcal{H}_e=\Gamma^{s+3}(TM)$, induces a right-invariant sub-Riemannian structure on $\mathcal{D}^{s+1}(M)$. Anticipating a bit, let $t\mapsto(\varphi(t),P(t))$ be a normal geodesic on $T^*\mathcal{D}^{s+1}(M)$ with $\varphi(0)=e$ and $P(0)\in\Gamma^{-s-1}(T^*M)\setminus\Gamma^{-s}(T^*M)$. 
Then, we claim that $\varphi(\cdot)$ is an elusive horizontal curve for the sub-Riemannian structure induced by $(\mathcal{H}_e,\left\langle\cdot,\cdot\right\rangle)$ on $\mathcal{D}^s(M)$. 

Indeed, it is clear that horizontal curves starting at $e$ coincide for both structures. 
Let us prove that $\varphi(\cdot)$ can have neither a normal nor an abnormal Lagrange multiplier, in the sub-Riemannian structure on $\mathcal{D}^s(M)$. If there would exist a normal Lagrange multiplier, then the corresponding initial momentum $P'(0)\in \Gamma^{-s}(T^*M)$ would satisfy $P(0)=P'(0)$ on $\mathcal{H}_e=\Gamma^{s+3}(TM)$ which is dense in $\Gamma^s(TM)$, and we would have $P'(0)=P(0)$, which is impossible. There cannot exist an abnormal Lagrange multiplier, because the range of $\d\,\mathrm{end}_{e}(\varphi(\cdot))$ contains $\mathcal{H}_{\varphi(1)}=T_{\varphi_1}\mathcal{D}^{s+3}(M)$, which is dense in $T_{\varphi_1}\mathcal{D}^s(M)$. This proves the claim.

\smallskip

Conversely, we can get new normal geodesics for a right-invariant sub-Riemannian structure on $\mathcal{D}^s(M)$ induced by a Hilbert subspace of vector fields with continuous inclusion in $\Gamma^{s+3}(TM)$, by restricting it to $\mathcal{D}^{s+1}(M)$, adding extra initial momenta by increasing the order of the cotangent bundle. In such a way, some elusive geodesics become normal geodesics. Note that it may happen that some elusive geodesics become abnormal curves because of the increased range in the choice of momenta. However, note also that this simple process does not turn every elusive geodesic into either a normal or an abnormal geodesic.

\smallskip

We conclude that, in accordance with \cite{Kurcyusz}, the choice of the cotangent space (and thus, the choice of the topology of the target space) is important. 
The stronger is the topology in the target space, and the larger is the dual, but then Lagrange multipliers become more and more irregular. It is therefore reasonable to avoid choosing a too strong topology in the target space.

In our setting there does not seem to exist a best possible choice for the cotangent bundle (better in the sense that, by adding new possibilities for the initial momenta, we would turn every possible elusive geodesic into either a normal or an abnormal geodesic). The question of finding a ``good" space of initial or final momenta (implying the absence of elusive geodesics) is open and seems quite difficult.%
\footnote{Note that these difficulties are also due to the fact that the sub-Riemannian problem consists of minimizing the action $A(\varphi(\cdot))$ over all horizontal curves $\varphi(\cdot)$ such that $\mathrm{end}_{\varphi_0}(\varphi(\cdot))=\varphi_1$, that is, of minimizing a functional under an infinite number of constraints. Assume that, instead, we consider the problem of minimizing the penalized functional $J(X(\cdot)) = \int_0^1\langle X(t),X(t)\rangle\,\d t+G(\mathrm{end}_{\varphi_0}(\varphi^X(\cdot)))$.
If $G$ is $\mathcal{C}^1$ and bounded below, then this (unconstrained) penalized problem has at least one solution $X(\cdot)$, and there exists a momentum mapping $t\mapsto P(t)$ such that $P(1)+\d G_{\varphi(1)}=0$ and such that $(\varphi(\cdot),P(\cdot))$ is solution of the normal geodesic Hamiltonian geodesic equations \eqref{eq_geod}.
This claim follows from an easy adaptation of the proofs of the results in the present paper (see also \cite{ATY}).
This framework can be used in order to ``approach" a target diffeomorphism $\varphi_1$ with an horizontal curve, by choosing a penalization function $G$.}
For instance, the authors found examples of Hilbert Lie groups (more precisely, $\ell^2(\N,\R^4)$, with $\R^4$ equipped with the Engel group structure) for which the natural space of momenta is not even locally convex.
\end{remark}

%
%

Following this preliminary discussion, we are now going to derive the normal and abnormal geodesic equations, which are Hamiltonian characterizations of non-elusive geodesics.
We stress again that, in general, no such Hamiltonian characterization can be derived for an elusive geodesic.

We will prove that any solution of the normal geodesic Hamiltonian equations, if it is well defined on $[0,1]$, projects onto a geodesic (critical point of the action). Therefore, instead of giving necessary conditions to minimize the action, we are rather going to provide sufficient conditions under which we have a non-elusive geodesic.

Hereafter, we first establish the normal geodesic equations, then the abnormal geodesic equations, and we finally provide necessary conditions for optimality.
The three theorems are then proved together.

\subsection{Normal geodesic equations}

Let $K_{\mathcal{H}_e}:\mathcal{H}_e^*\rightarrow\mathcal{H}_e$ be the inverse of the operator $X\mapsto\langle X,\cdot\rangle$. By analogy with the classical Riemannian case, we call this operator the \textit{sub-musical isomorphism}.
We define the $\mathcal{C}^k$ vector bundle morphism $K_{\mathcal{H}_\varphi}: T^*_\varphi\mathcal{D}^s(M)\rightarrow T_\varphi\mathcal{D}^s(M)$ by $K_{\mathcal{H}_\varphi} = \d R_\varphi \, K_{\mathcal{H}_e} \, \d R_{\varphi}^*$
for every $\varphi\in\mathcal{D}^s(M)$. 

\begin{definition}
We define the \emph{normal Hamiltonian} $h:T^*\mathcal{D}^s(M)\rightarrow\R$ by
\begin{equation}\label{def_normalHam}
h(\varphi,P)=\frac{1}{2} P(K_{\mathcal{H}_\varphi }P),
\end{equation}
for every $(\varphi,P)\in T^*\mathcal{D}^s(M)$.
\end{definition}

The expression \eqref{def_normalHam} means that, as in classical sub-Riemannian geometry (see \cite{MBOOK}), $h(\varphi,P)$ is the squared norm of $P$ for the cometric induced on $T^*\mathcal{D}^s(M)$ by the sub-Riemannian structure. This is the usual way to define the normal Hamiltonian.

In canonical coordinates, we have $K_{\mathcal{H}_\varphi }P=X(\varphi,P)\circ\varphi$, with $X(\varphi,P)=K_{\mathcal{H}_e}(\d R_{\varphi})^*.P$, and hence 
$$
h(\varphi,P)=\frac{1}{2}P(X(\varphi,P)\circ\varphi)=\frac{1}{2}\langle X(\varphi,P),X(\varphi,P)\rangle.
$$
Note that $(\varphi,P)\mapsto X(\varphi,P)$ is of class $\mathcal{C}^k$. 

Denoting by $\omega$ the canonical strong symplectic form on $T^*\mathcal{D}^s(M)$, the symplectic gradient $\nabla^\omega h:T^*\mathcal{D}^s(M)\rightarrow  TT^*\mathcal{D}^s(M)$ of $h$ (which is of class $\mathcal{C}^{k-1}$) is defined by the relation $\d h=\omega(\nabla^\omega h,\cdot)$. In canonical coordinates, we have $\nabla^\omega h =(\partial_Ph,-\partial_\varphi h)$, where $\partial_Ph\in T_\varphi^{**}\mathcal{D}^s(M)=T_\varphi\mathcal{D}^s(M)$ thanks to the natural isomorphism between a Hilbert space and its bidual space, and we have
$$
\nabla^\omega h(\varphi,P)= \left(K_{\mathcal{H}_\varphi} P,-(\partial_\varphi K_{\mathcal{H}_\varphi} P)^*.P\right)
= (X(\varphi,P)\circ\varphi,-(\partial_\varphi(X(\varphi,P)\circ\varphi))^*.P) .
$$
%

\begin{theorem}\label{normalgeod}
We assume that $k\geq 2$. Then the symplectic gradient of $h$ is of class $\mathcal{C}^{k-1}$ and admits a global flow: for every $\varphi_0\in \mathcal{D}^s(M)$ and every $P_0\in T_{\varphi_0}^*\mathcal{D}^s(M)$, there is a unique global solution $(\varphi(\cdot),P(\cdot)): \R \rightarrow T^*\mathcal{D}^s(M)$ (meaning that $P(t)\in T_{\varphi(t)}^*\mathcal{D}^s(M)$ for every $t$) of 
\begin{equation}\label{eq_geod}
(\dot{\varphi}(t),\dot{P}(t)) = \nabla^\omega h(\varphi(t),P(t)), \quad t\in\R,
\end{equation}
and such that $(\varphi(0),P(0))=(\varphi_0,P_0)$. This global flow is of class $H^1$ with respect to $t$, and of class $\mathcal{C}^{k-1}$ with respect to the initial conditions $(\varphi_0,P_0)$.
Moreover, $\varphi(\cdot)$ is a geodesic on any sub-interval of $\R$, which implies that the norm of its logarithmic velocity is constant.
\end{theorem}

\begin{definition}
In the conditions of Theorem \ref{normalgeod}, $\varphi(\cdot)$ is said to be a \emph{normal geodesic}, the couple $(\varphi(\cdot),P(\cdot))$ is said to be a \emph{normal extremal lift} of $\varphi(\cdot)$, and $P(\cdot)$ is called a \emph{covector}.
\end{definition}


Theorem \ref{normalgeod} says that, if $k\geq2$, then $\mathcal{D}^s(M)$ admits a global normal geodesic flow, of class $H^1$ in time and $C^{k-1}$ in the initial conditions. Note that, since $h$ is of class $\mathcal{C}^k$, it is already clear that $\nabla^\omega h$ is of class $\mathcal{C}^{k-1}$ and thus admits a unique maximal flow. The fact that integral curves of $\nabla^\omega h$ project onto geodesics will be proved in Section \ref{proofgeod}, and the global property of the flow will be proved thanks to the momentum formulation stated in Section \ref{momentum}.

\begin{remark}\label{rem_locexpression}
In canonical coordinates, the normal geodesic equations \eqref{eq_geod} are written as
$$
\dot{\varphi}(t) = K_{\mathcal{H}_{\varphi(t)} } P(t),\quad
\dot{P}(t) = - (\partial_{\varphi}K_{\mathcal{H}_{\varphi(t)} }  P(t))^*.P(t).
$$
Note that, if $K$ is the reproducing kernel associated with $\mathcal{H}_e$ (see Remark \ref{rkhs}), then we have
$$
\begin{aligned}
\partial_t \varphi(t,x)&=\int_M K(\varphi(t,x),\varphi(t,y))P(t,y)\,dy_g,\\
\partial_t P(t,x)&=-P(t,x)\int_M \partial_1K(\varphi(t,x),\varphi(t,y))P(t,y)\,dy_g,
\end{aligned}
$$
for every $x\in M$. 

Note that these equations are not partial differential equations or integro-differential equations. They are ordinary differential equations whose terms are $\mathcal{C}^1$ and non-local.
The main interest of this formulation is that the geodesic equations are completely explicit, and can be implemented numerically with relative ease and efficiency. The computation of the reproducing kernel of $\mathcal{H}_e$, which is no easy task, is not required. However, in many cases, particularly in shape analysis, $\mathcal{H}_e$ is not given explicitly: instead, it is defined through an explicit kernel, which simplifies matters greatly (see \cite{ATY,TY2,YBOOK} and Section \ref{ex}).
\end{remark}

\begin{remark}
According to the notations above, the logarithmic velocity of a normal geodesic $\varphi(\cdot)$ is given by 
$$X(\cdot)=X(\varphi(\cdot),P(\cdot))=K_{\mathcal{H}_e}(\d R_{\varphi(\cdot)})^*.P(\cdot).$$
\end{remark}

\begin{remark}
If $\mathcal{H}_e$ has a continuous inclusion in $\Gamma^s(TM)$ for every $s\in\N$ (which implies that $\mathcal{H}_e$ has a continuous injection in the Fr\'echet space of smooth vector fields), then, since any compactly supported co-current $P$ (that is, any one-form with distributional coefficients) belongs to $\Gamma^{-s}(T^*M)$ for some $s\in\N$, it follows that any such $P$ generates a locally minimizing normal geodesic starting at $e$.
Therefore, the Fr\'echet Lie group $\mathcal{D}^\infty(M)=\cap_{s> {d}/{2}+1}\mathcal{D}^s(M)$ inherits of a strong right-invariant sub-Riemannian structure.
\end{remark}

\subsection{Abnormal geodesic equations}
The abnormal geodesic equations actually provide as well a Hamiltonian characterization of singular curves.

\begin{definition}
We define the \emph{abnormal Hamiltonian} $H^0: T^*\mathcal{D}^s(M)\times \mathcal{H}_e \rightarrow \R$ by
\begin{equation*}
H^0(\varphi,P,X) = P(X\circ\varphi)=P(\d R_\varphi . X) .
\end{equation*}
\end{definition}
Since $X$ is of class $H^{s+k}$, it follows that $H^0$ is of class $\mathcal{C}^k$.
We have $\partial_XH^{0}(\varphi,P,X)=(\d R_\varphi)^*.P$, where the partial derivative $\partial_XH^{0}: T^*\mathcal{D}^s(M) \times \mathcal{H}_e\rightarrow \mathcal{H}_e^*$ is understood as a partial derivative along the fibers of a vector bundle.
The symplectic gradient $\nabla^\omega H^0:T^*\mathcal{D}^s(M)\times \mathcal{H}_e\rightarrow  TT^*\mathcal{D}^s(M)$ of $H^0$, defined by the relation $\d H^0=\omega(\nabla^\omega H^0,\cdot)$, is given in canonical coordinates $(\varphi,P)$ on $T^*\mathcal{D}^s(M)$ by $\nabla^\omega H^0(\varphi,P,X)=(X\circ\varphi,-(\partial_\varphi(X\circ\varphi))^*.P)$.

\begin{theorem}\label{thm_abnormal}
We assume that $k\geq 2$. 
Let $\varphi(\cdot)\in H^1(0,1;\mathcal{D}^s(M))$ be an horizontal curve with logarithmic velocity $X(\cdot)=\dot{\varphi}(\cdot)\circ\varphi(\cdot)^{-1}$. Then $\varphi(\cdot)$ is a singular curve if and only if there exists a mapping $P(\cdot)$ on $[0,1]$, {of class $H^1$ in time}, such that $P(t)\in T^*_{\varphi(t)}\mathcal{D}^s(M)\setminus\{0\}$ and
\begin{align}
& (\dot{\varphi}(t),\dot{P}(t)) = \nabla^\omega H^0(\varphi(t),P(t),X(t)) ,  \label{abgeod1}   \\
& \partial_XH^0(\varphi(t),P(t),X(t))=(\d R_{\varphi(t)})^*.P(t)=0,  \label{abgeod2}
\end{align}
for almost every $t\in[0,1]$.
\end{theorem}

\begin{definition}
In the conditions of Theorem \ref{thm_abnormal}, the couple $(\varphi(\cdot),P(\cdot))$ is said to be an \emph{abnormal lift} of the singular curve $\varphi(\cdot)$, and $P(\cdot)$ is said to be a \emph{singular covector}.
\end{definition}

%
%


\subsection{Necessary conditions for optimality}
The following result is an extension of the usual Pontryagin maximum principle (see \cite{P}) to our specific infinite-dimensional setting.

\begin{theorem}\label{thm_pmp}
We assume that $k\geq 1$. 
Let $\varphi(\cdot)\in H^1(0,1;\mathcal{D}^s(M))$ be a minimizing horizontal curve with logarithmic velocity $X(\cdot)=\dot{\varphi}(\cdot)\circ\varphi(\cdot)^{-1}$. Then $\varphi(\cdot)$ is a geodesic, and:
\begin{itemize}
\item either $\varphi(\cdot)$ is a normal geodesic, and in that case, it is the projection onto $\mathcal{D}^s(M)$ of a normal extremal lift $(\varphi(\cdot),P(\cdot))$ on $[0,1]$ (satisfying \eqref{eq_geod} on $[0,1]$);
\item or $\varphi(\cdot)$ is a singular curve, and in that case, it is the projection onto $\mathcal{D}^s(M)$ of an abnormal extremal lift $(\varphi(\cdot),P(\cdot))$ on $[0,1]$ (satisfying \eqref{abgeod1}-\eqref{abgeod2} almost everywhere on $[0,1]$);
\item or $\varphi(\cdot)$ is elusive.
\end{itemize}
\end{theorem}

\begin{remark}
In finite dimension, it has been established in \cite{RT} and in \cite{Ag} that the set of end-points of normal geodesics is an open dense subset of the ambient manifold. Although such a result is not established in our infinite-dimensional context (it is all the more difficult than one has also to deal with elusive curves), it is however expected that, in some appropriate sense, the "generic" case of the above theorem is the first one (normal geodesics). 

Besides, it has been established in \cite{ChitourJeanTrelatJDG,ChitourJeanTrelatSICON} that, in finite-dimensional sub-Riemannian geometry, for generic (in a strong Whitney sense) horizontal distributions of rank greater than or equal to three, any singular curve cannot be minimizing. 
Although such a result seems currently out of reach in our infinite-dimensional setting, we do expect that, since our distribution is infinite-dimensional, there is no minimizing singular curve for generic distributions. 
\end{remark}

\subsection{Momentum formulation: sub-Riemannian Euler-Arnol'd equation}\label{momentum}
We define the \emph{momentum map} $\mu:T^*\mathcal{D}^s(M)\rightarrow\Gamma^{-s}(T^*M)$ by $\mu(\varphi,P)=(\d R_\varphi)^*.P$.

%

\begin{proposition}\label{prop_momentum}
We assume that $k\geq1$. Let $\varphi(\cdot)\in H^1(0,1;\mathcal{D}^s(M))$ be either a normal geodesic or a singular curve, with logarithmic velocity $X(\cdot)\in L^2(0,1;\mathcal{H}_e)$, and let $(\varphi(\cdot),P(\cdot))$ be an extremal lift (either normal or abnormal) of $\varphi(\cdot)$. We denote by $\mu(t)=\mu(\varphi(t),P(t))$ the corresponding momentum along the extremal, which is continuous in time.

Then the curve $\mu(\cdot)$ has Sobolev class $H^1$ in the coarser space $\Gamma^{-s-1}(T^*M)$, with derivative given almost everywhere by 
\begin{equation}\label{EulerArnol'd}
\dot{\mu}(t)=\mathrm{ad}^*_{X(t)}\mu(t)=-\mathcal{L}_{X(t)}\mu(t),
\end{equation}
for almost every $t\in[0,1]$.
Here, $\mathrm{ad}_X:\Gamma^{s+1}(TM)\rightarrow\Gamma^{s}(TM)$, with $\mathrm{ad}_XY=[X,Y]$, and $\mathcal{L}_X$ the Lie derivative with respect to $X$. As a consequence, we have
 $$
\mu(t)=\varphi(t)_*\mu(0),
$$
for every $t\in[0,1]$.
\end{proposition}

\begin{proof} 
Let $Y\in \Gamma^{s+1}(TM)\subset\Gamma^s(TM)$ and let $t\in[0,1]$. Then, in canonical coordinates, we have
$$
\mu(t)(Y)=P(t)(Y\circ\varphi(t))
=P(0)(Y)+\int_0^t \left( P(\tau)(\d Y\circ\varphi(\tau).X\circ\varphi(\tau))+\dot{P}(\tau)(Y\circ\varphi(\tau)) \right) \d \tau .
$$ 
Since the derivative of the covector is given, in both normal and singular cases, by
$$\dot{P}(\tau)(Y\circ\varphi(\tau))=-P(\tau)(\partial_\varphi(X(\tau)\circ\varphi(\tau)).Y\circ\varphi(\tau)) 
= -P(\tau)((\d X(\tau).Y)\circ\varphi(\tau)),$$ 
for every $\tau\in[0,t]$, we infer that
$$
\begin{aligned}
{\mu}(t)(Y)&=P(0)(Y)+\int_0^t \left( P((dY.X(\tau)-dX(\tau).Y)\circ\varphi) \right) \d \tau \\
&=P(0)(Y)+\int_0^t P(\tau)([X(\tau),Y]\circ\varphi(\tau)) \, \d \tau\\
&=P(0)(Y)+\int_0^t\mu(\tau)([X(\tau),Y]) \, \d \tau\\
&=P(0)(Y)+\int_0^t\mathrm{ad}_{X(\tau)}^*\mu(\tau)(Y) \, \d \tau .
\end{aligned}
$$
Note that $t\mapsto\mathrm{ad}_{X(t)}^*\mu(t)$ belongs to $L^2(0,1;\Gamma^{-s-1}(T^*M))$. Indeed, the Lie bracket of vector fields yields a continuous bilinear mapping $(\Gamma^{s+1}(TM))^2\rightarrow\Gamma^s(TM)$, and thus 
$$
\vert \mathrm{ad}_{X(t)}^*\mu(t)(Y)\vert^2\leq C \left( \max_{t\in [0,1]}\vert \mu(t)\vert_{-s}\right)^2\langle X\rangle\vert Y\vert_{s+1}^2.
$$
Here, the notation $\vert\cdot\vert_{-s}$ stands for the usual operator norm on the dual space $\Gamma^{-s}(T^*M)=(\Gamma^s(TM))^*$.

Since $k\geq1$ and $\varphi(\cdot)$ is horizontal, we have $\varphi(\tau)\in\mathcal{D}^{s+1}(M)$ for every $\tau\in[0,t]$, hence, from the above equation, we easily infer that $\mu(t)=\varphi(t)_*\mu(0)$ on $\Gamma^{s+1}(TM)$, which is dense in $\Gamma^{s}(TM)$. The proposition follows.
\end{proof}

\begin{remark}
Since $k\geq1$, any normal geodesic $\varphi(\cdot)$ is of class $\mathcal{C}^1$. Moreover, $X(t)=K_{\mathcal{H}_e}\mu(t)$ and we recover the classical formula for critical points of the action on Lie groups (see \cite{HMR,MRBOOK})
$$
\dot{\mu}(t)=\mathrm{ad}^*_{K_{\mathcal{H}_e}\mu(t)}\mu(t).
$$
The differential equation \eqref{EulerArnol'd} is the generalization of the famous Euler-Arnol'd equation to our sub-Riemannian setting.

It can be noted that, in the Riemannian case, and for smooth vector fields, we have $K_{\mathcal{H}_e} \mathrm{ad}^*_{X} = \mathrm{ad}^T_{X} K_{\mathcal{H}_e}$, and it follows that the above equation is equivalent to
\begin{equation}\label{EPDiff}
\dot{X}(t)=\mathrm{ad}^T_{X(t)}X(t),
\end{equation}
where $\mathrm{ad}^T_{X}$ is the transpose of the operator $\mathrm{ad}_X$ with respect to the Hilbert product $\langle\cdot,\cdot\rangle$.
As is well known, we then obtain the Euler equation for the weak $L^2$ metric (see \cite{A,EM}) on vector fields with divergence zero, and with other metrics we obtain other equations, such as KdV, Camassa-Holm (see the survey paper \cite{BBM}). Let us note that, if $M=\R^d$, then \eqref{EulerArnol'd} is equivalent to
$$
\partial_t\mu(t) = -(X(t).\nabla)\mu(t) - (\mathrm{div}(X(t))\mu(t) - (\d X(t))^*\mu(t),
$$
which has the same form as that of the usual EPDiff equation.
In some sense, the differential equation \eqref{EPDiff} is the version on the tangent space of the differential equation \eqref{EulerArnol'd}, which lives on the cotangent space.

In the sub-Riemannian framework of the present paper, we \textit{cannot} write the differential equation \eqref{EPDiff} on the tangent space, because $\mathcal{H}_e$ is usually not a subspace of $\Gamma^s(TM)$ that is invariant under $\mathrm{ad}_X$ (and then, we do not have $K_{\mathcal{H}_e} \mathrm{ad}^*_{X} = \mathrm{ad}^T_{X} K_{\mathcal{H}_e}$). This same restriction already appears in the finite-dimensional setting where the geodesic equations can only be written in the cotangent space, but not in the tangent space.
\end{remark}

Because of the loss of a derivative, it is harder to prove the existence of solutions of the differential equation $\dot{\mu}=\mathrm{ad}^*_{K_{\mathcal{H}_e}\mu}\mu$ without using Theorem \ref{normalgeod}. On the other hand, once the existence and uniqueness of the geodesic flow is ensured, the momentum formulation can be used to find various  quantities conserved by the geodesic flow. For example, we have the following result.

\begin{corollary}\label{consord}
{In the context of Proposition \ref{prop_momentum}}, the support of $P(\cdot)$, and its order of regularity\footnote{Here, we refer to the regularity of the coefficients that appear in the  (distributional-valued) $1$-form with which $P(0)$ is identified. For example, if $P(0)$ is identified with a $1$-form whose coefficients belong to the space of Radon measures on $M$, then the same holds for $P(t)$, for every time $t$.} up to Sobolev class $H^{s+k-1}$, are preserved on $[0,1]$. 
\end{corollary}

\begin{proof}
This is an immediate consequence of the formula $\mu(t)=\varphi(t)_*\mu(0)$, using the fact that $\varphi(t)\in \mathcal{D}^{s+k}(M)$ as a flow of vector fields of class $H^{s+k}$. Hence, the order of regularity of $\mu$ up to Sobolev class $H^{s+k}$ is obviously constant along the curve, and the support of $\mu(t)$ is the image by $\varphi(t)$ of the support of $\mu(0)$. Using that $P=(\d R_\varphi)_*\mu$, the result follows.
\end{proof}

\subsection{Proof of Theorems \ref{normalgeod}, \ref{thm_abnormal} and \ref{thm_pmp}}\label{proofgeod}
Let us first compute the adjoint of the derivative of the end-point mapping.

\begin{lemma}\label{adjend}
Let $\varphi(\cdot)\in \Omega_{\varphi_0}$ be an horizontal curve with logarithmic velocity $X(\cdot)=\dot{\varphi}(\cdot)\circ\varphi(\cdot)^{-1}$. We set $\varphi_1=\varphi(1)=\mathrm{end}_{\varphi_0}(\varphi(\cdot))$. For every $P_{\varphi_1}\in T^*_{\varphi_1}\mathcal{D}^s(M)$, the pull-back $(\d\,\mathrm{end}_{\varphi_0}(\varphi(\cdot)))^*.P_{\varphi_1}$ of $P_{\varphi_1}$ by $\d\, \mathrm{end}_{\varphi_0}(\varphi(\cdot))$ is a continuous linear form on $L^2(0,1;\mathcal{H}_e)$, and can therefore be identified to an element of $L^2(0,1;\mathcal{H}_e)^*=L^2(0,1;\mathcal{H}_e^*)$, given by
\begin{equation}\label{dendetoile}
((\d\,\mathrm{end}_{\varphi_0}(\varphi(\cdot)))^*.P_{\varphi_1})(t)=\partial_XH^0(\varphi(t),P(t),X(t))=(\d R_{\varphi(t)})^*.P(t),
\end{equation}
for almost every $t\in[0,1]$, where $(\varphi(\cdot),P(\cdot)):[0,1]\rightarrow T^*\mathcal{D}^s(M)$ is the unique absolutely continuous mapping solution of $(\dot{\varphi}(t),\dot{P}(t))=\nabla^\omega H^0(\varphi(t),P(t),X(t))$ on $[0,1]$ and $P(1)=P_{\varphi_1}$.
\end{lemma}

\begin{proof}[Proof of Lemma \ref{adjend}.]
In local coordinates, the fibered part of the differential equation of the lemma is $
\dot{P}(t)=-(\partial_\varphi(X(t)\circ\varphi(t)))^*.P(t)$, which is a linear differential equation. The Cauchy-Lipschitz theorem for linear differential equations therefore ensures global existence and uniqueness of a solution $P(\cdot)$ of class $H^1$ (and thus, absolutely continuous) such that $P(1)=P_{\varphi_1}$. Let us now prove the formula \eqref{dendetoile}.
The mapping $X(\cdot)\mapsto \varphi^X(\cdot)$ is of class $\mathcal{C}^k$, with $k\geq1$, and is defined implicitly by the differential equation $\dot{\varphi}^X(t)-X(t)\circ\varphi^X(t)=0$ for almost every $t\in[0,1]$, with $\varphi^X(0)=\varphi_0$.
To compute its derivative $\delta \varphi(\cdot)=\d \varphi^X(X(\cdot)).\delta X(\cdot)$ in the direction $\delta X(\cdot)\in L^2(0,1;\mathcal{H}_e)$, we differentiate this differential equation in local coordinates, and obtain that $\delta\dot\varphi(t)-\delta X(t)\circ\varphi^X(t)-\partial_\varphi(X(t)\circ{\varphi^X(t)}).\delta\varphi(t)=0$ for almost every $t\in[0,1]$, with $\delta \varphi(0)=0$. For every $\delta X(\cdot)\in L^2(0,1;\mathcal{H}_e)$, this Cauchy problem has a unique solution $\delta \varphi(\cdot)$, and we have $\d\,\mathrm{end}_{\varphi_0}(\varphi(\cdot)).\delta X(\cdot)=\delta \varphi(1)$.
Moreover, we have
\begin{equation*}
\begin{aligned}
\int_0^1\partial_XH^{0}(\varphi(t),P(t),X(t)).\delta X(t)\,\d t
&=\int_0^1P(t)(\delta X(t)\circ\varphi(t))\, \d t\\
&=\int_0^1P(t)(\delta\dot\varphi(t))\, \d t - \int_0^1P(t)\left( \partial_\varphi(X(t)\circ{\varphi(t)}).\delta\varphi(t) \right)\d t\\
&=\int_0^1 \left( P(t)(\delta\dot\varphi(t))  + \dot{P}(t)\left(\delta\varphi(t)\right)\right) \d t \\
&=P_{\varphi_1}(\delta \varphi(1)) = (\d\,\mathrm{end}_{\varphi_0}(\varphi(\cdot)))^*.P_{\varphi_1},
\end{aligned}
\end{equation*}
which yields \eqref{dendetoile}.
\end{proof}

Theorem \ref{thm_abnormal} follows from Lemma \ref{adjend} because $\varphi(\cdot)$ is a singular curve if and only if then there exists $P_{\varphi_1}\in T^*_{\varphi_1}\mathcal{D}^s(M)\setminus\{0\}$ such that $(\d\,\mathrm{end}_{\varphi_0}(\varphi(\cdot)))^*.P_{\varphi_1}=0$.

Let $\varphi_0$ and $\varphi_1$ be two elements of $\mathcal{D}^s(M)$, and let $\varphi(\cdot)\in\Omega_{\varphi_0,\varphi_1}$. 
Since $\Omega_{\varphi_0,\varphi_1}=\mathrm{end}_{\varphi_0}^{-1}(\{\varphi_1\})$, we have $T_{\varphi(\cdot)}\Omega_{\varphi_0,\varphi_1}\subset\ker (\d\,\mathrm{end}_{\varphi_0}(\varphi(\cdot)))$ (see Definition \ref{def_geod} for the definition of the set of all infinitesimal variations). Note that, if $\d\,\mathrm{end}_{\varphi_0}(\varphi(\cdot))$ were surjective, then $\Omega_{\varphi_0,\varphi_1}$ would be, locally at $\varphi(\cdot)$, a $\mathcal{C}^k$ submanifold of $\Omega_{\varphi_0}$, and then $T_{\varphi(\cdot)}\Omega_{\varphi_0,\varphi_1} = \ker (\d\,\mathrm{end}_{\varphi_0}(\varphi(\cdot)))$. But, as already said, in our context only the inclusion is true.

Since $T_{\varphi(\cdot)}\Omega_{\varphi_0,\varphi_1}\subset\ker (\d\,\mathrm{end}_{\varphi_0}(\varphi(\cdot)))$, if there is some $P_{\varphi_1}\in T^*_{\varphi_1}\mathcal{D}^s(M)$ such that $\d A(\varphi(\cdot))=(\d\,\mathrm{end}_{\varphi_0}(\varphi(\cdot)))^*.P_{\varphi_1}$,
then $\varphi(\cdot)$ is a critical point of $A$ restricted to $\Omega_{\varphi_0,\varphi_1}$ (and hence $\varphi(\cdot)$ is a geodesic steering $\varphi_0$ to $\varphi_1$). 
Conversely, according to the discussion done at the beginning of Section \ref{sec_geod}, this Lagrange multiplier relation is satisfied whenever $\varphi(\cdot)$ is a geodesic steering $\varphi_0$ to $\varphi_1$ which is neither singular nor elusive.

Besides, the differential $\d A(\varphi(\cdot))\in L^2(0,1;\mathcal{H}_e^*)$ is given by $\d A(\varphi)(t)=\langle X(t),\cdot\rangle$.
It follows from Lemma \ref{adjend} that
$$
\left( \d A(\varphi(\cdot))-\d\, \mathrm{end}_{\varphi_0}(\varphi(\cdot))^*.P_{\varphi_1}\right)(t)=\langle X(t),\cdot\rangle-\partial_XH^0(\varphi(t),P(t),X(t)),
$$
for almost every $t\in[0,1]$, where $P(\cdot):[0,1]\rightarrow T^*\mathcal{D}^s(M)$ is the unique solution of $(\dot{\varphi}(t),\dot{P}(t))=\nabla^\omega H^0(\varphi(t),P(t),X(t))$ on $[0,1]$ and $P(1)=P_{\varphi_1}$.
Defining the total Hamiltonian by
$$
\begin{array}{rcl}
H: T^*\mathcal{D}^s(M)\times \mathcal{H}_e & \rightarrow & \R \\
(\varphi,P,X) & \mapsto & P(X\circ\varphi)-\frac{1}{2}\langle X,X\rangle ,
\end{array}
$$
we have $\nabla^{\omega}H=\nabla^{\omega}H^0$, and
$\left(\d A(\varphi(\cdot))-(\d\, \mathrm{end}_{\varphi_0}(\varphi(\cdot)))^*.P_{\varphi_1}\right)(t)=-\partial_XH(\varphi(t),P(t),X(t))$
for almost every $t\in[0,1]$.
We have obtained the following lemma.

\begin{lemma}
Let $\varphi_0\in \mathcal{D}^s(M)$.
Let $X\in L^2(0,1;\mathcal{H}_e)$ be the logarithmic velocity of an horizontal curve $\varphi(\cdot)\in H^1(0,1;\mathcal{D}^s(M))$ starting at $\varphi_0$. The two following statements are equivalent:
\begin{itemize}
\item There exists an absolutely continuous fibered mapping $(\varphi(\cdot),P(\cdot)):[0,1]\rightarrow T^*\mathcal{D}^s(M)$ such that
$(\dot{\varphi}(t),\dot{P}(t))=\nabla^\omega H(\varphi(t),P(t),X(t))$
for almost every $t\in[0,1]$ and $P(1)=P_{\varphi_1}$, and such that
$$
0=\partial_XH(\varphi(t),P(t),X(t))=(\d R_{\varphi(t)})^*.P(t)-\langle X(t),\cdot\rangle,
$$
for almost every $t\in [0,1]$.
\item There exists $P_{\varphi_1}\in T_{\varphi_1}^*\mathcal{D}^s(M)\setminus\{0\}$  such that $\d A(\varphi(\cdot))=(\d\,\mathrm{end}_{\varphi_0}(\varphi(\cdot)))^*.P_{\varphi_1}$.
\end{itemize}
Under any of those statements, $\varphi(\cdot)$ is a (normal) geodesic.
\end{lemma}

For every fixed $(\varphi,P)\in T^*\mathcal{D}^s(M)$, the equation $\partial_XH(\varphi,P,X)=0$ yields $\langle X,\cdot\rangle=(\d R_{\varphi})^*.P$, whose unique solution is given by $X(\varphi,P)=K_{\mathcal{H}_e}(\d R_{\varphi})^*.P$. 
Then we obtain the normal Hamiltonian $h:T^*\mathcal{D}^s(M)\rightarrow\R$ by setting
$$
h(\varphi,P) = H(\varphi,P,X(\varphi,P)) = \frac{1}{2}P(X(\varphi,P)\circ\varphi)=\frac{1}{2}\langle X(\varphi,P),X(\varphi,P)\rangle.
$$
Theorems \ref{normalgeod} and \ref{thm_pmp} follow, except for the global property of the flow in Theorem \ref{normalgeod}.

\begin{lemma}
The geodesic flow defined in Theorem \ref{normalgeod} is global.
\end{lemma}

\begin{proof}
Let $(\varphi(\cdot),P(\cdot)):I=(a,b)\rightarrow T^*\mathcal{D}^s(M)$ be the maximal solution to the Cauchy problem $(\dot{\varphi}(t),\dot P(t))=\nabla^\omega h(\varphi(t),P(t))$, $(\varphi(0),P(0)=(\varphi_0,P_0)$, with $0\in I$. Since $\nabla^\omega h$ has a well-defined maximal flow when $k\geq2$, it suffices to prove that, if $b<+\infty$, then $(\varphi(t),P(t))$ converges to a limit $(\varphi_b,P_b)$. 

Let $X(\cdot)$ be the logarithmic velocity of $\varphi(\cdot)$ and, for $t\in I$, let $\mu(t)=(\d R_\varphi(t))^*P(t)$. Then, we have $\mu(t)=\varphi(t)_*\mu(0)$ for every $t\in I$. Since $\varphi(\cdot)$ is a geodesic, we have $\Vert X(t)\Vert=\Vert X(0)\Vert\leq C\Vert X(0)\Vert_{s+k}$ for every $t\in I$, which immediately implies that $\varphi(t)\underset{t\rightarrow b}{\longrightarrow}\varphi_b$ in the topology of $\mathcal{D}^{s+k}(M)$ for some $\varphi_b\in \mathcal{D}^{s+k}(M)$. On the other hand, the mapping $\mathcal{D}^{s+1}(M)\times \Gamma^{-s}(T^*M)\rightarrow\Gamma^{-s}(T^*M)$ defined by $(\varphi,\mu)\mapsto \varphi_*\mu$ is continuous, so that $\mu(t)=\varphi(t)_*\mu(0)\underset{t\rightarrow b}{\longrightarrow}\mu_b={\varphi_b}_*\mu(0)$ and $P(t)\underset{t\rightarrow b}{\longrightarrow}P_b$, with $\mu_b=\d R_{\varphi_b}^*P_b$.

The same argument shows that $a=-\infty$.
\end{proof}

\section{Examples of geodesic equations}\label{ex}

\subsection{Normal geodesic equations in $\mathcal{D}(\R^d)$}
In this section, we assume that $M=\R^d$.
Let $s_0$ be the smallest integer such that $s_0>d/2$. It is easy to prove that for every integer $s\geq s_0+1$, the group $\mathcal{D}^s(\R^d)$ coincides with the set of diffeomorphisms $\varphi$ of $\R^d$ such that $\varphi-\mathrm{Id}_{\R^d}\in H^s(\R^d,\R^d)$, and is an open subset of the affine Hilbert space $\mathrm{Id}_{\R^d}+H^s(\R^d,\R^d)$ (endowed with the induced topology, see Section \ref{sec_defi}).

Since we are in $\R^d$, we have $T\mathcal{D}^{s}(\R^d)=\mathcal{D}^{s}(\R^d)\times H^s(\R^d,\R^d)$ and $T^*\mathcal{D}^{s}(\R^d)=\mathcal{D}^{s}(\R^d)\times H^{-s}(\R^d,(\R^d)^*)$.
Therefore, a covector $P\in T_\varphi^*\mathcal{D}^{s}(\R^d)$ is a one-form on $\R^d$ with distributional coefficients in $H^{-s}(\R^d)$, denoted by $P=P_1dx^1+\dots+P_ddx^d=(P_1,\dots,P_d)$. For a vector field $X$, we can also write $X=X^1e_1+\dots+X^de_d=(X^1,\dots,X^d)$, with $(e_i)$ the canonical frame of $\R^d$. 

The Euclidean inner product of two vectors $v$ and $w$ of $\R^d$ is denoted by $v\cdot w$. The notation $v^T$ stands for for the linear form $w\mapsto c\cdot w$. Conversely, for a linear form $p\in(\R^d)^*$, we denote by $p^T$ the unique vector $v$ in $\R^d$ such that $p=v^T$.

These notations are extended to vector fields and to $1$-forms with distributional coefficients, by setting $(X^1e_1+\dots+X^de_d)^T=X^1dx^1+\dots+X^ddx^d$ and $(P_1dx^1+\dots+P_ddx^d)^T=P_1e_1+\dots+P_de_d$.

\subsubsection{Spaces of vector fields with Gaussian kernels}
Let $\mathcal{H}_e$ be the Hilbert space of vector fields on $\R^d$ associated with the reproducing kernel $K:\R^d\times\R^d\rightarrow \mathrm{End}((\R^d)^*,\R^d)$ defined by
$K(x,y)p=e(x,y)p^T$, for every $p\in (\R^d)^*$, with $e(x-y)=e^{-\frac{\vert x-y\vert^2}{2\sigma}}$, for some $\sigma>0$. This space is widely used in shape deformation analysis (see \cite{TY1,TY2,YBOOK}).
Note that the mapping $(x,y)\mapsto e(x-y)$ is equal (up to a multiplying scalar) to the heat kernel at time $\sigma$.

The elements of $\mathcal{H}_e$ are analytic, and all their derivatives decrease exponentially at infinity, hence $\mathcal{H}_e\subset H^s(\R^d,\R^d)$ for every $s\in\N$. Moreover, for every one-form with tempered distributional coefficients $P$, the vector field $X$ such that $\langle X,\cdot\rangle=P$ on $\mathcal{H}_e$ is given by
$
X(x)=\int_{\R^d}e(x-y)P(y)^T\,\d y,
$
for every $x\in\R^d$.
For every $(\varphi,P)\in T^*\mathcal{D}^{s}(\R^d)= \mathcal{D}^{s}(\R^d)\times H^{-s}(\R^d,\R^d)$, the solution of $\partial_XH(\varphi,P,X)=0$ is $X(\varphi,P)(x)=\int_{\R^d}e(x-\varphi(y))P(y)^T\,\d y$.
Therefore, the normal Hamiltonian is given by
$$
h(\varphi,P) = \frac{1}{2}P(X(\varphi,P)\circ\varphi)
=  \frac{1}{2} \int_{\R^d\times\R^d}e(\varphi(x)-\varphi(y)) P(x)\cdot P(y)\,\d y\,\d x .
$$
Since $e(x-y)=e(y-x)$ and $\d e(x).v=-\frac{1}{\sigma}e(x)(x\cdot v)$, we get that the normal geodesic equations are written in the distributional sense as
\begin{equation*}
\begin{aligned}
\partial_t{\varphi}(t,x)&=\int_{\R^d}e(\varphi(t,x)-\varphi(t,y))P(t,y)^T\,\d y ,\\
\partial_t{P}(t,x)& =\frac{1}{\sigma} \int_{\R^d\times\R^d} e(\varphi(t,x)-\varphi(t,y)) \left(\varphi(t,x)-\varphi(t,y)\right)^T P(t,x)\cdot P(t,y) \,\d y.
\end{aligned}
\end{equation*}
%
A particularly simple example is when $P=a\otimes\delta_{x_0}$, with $a\in(\R^d)^*$ (that is,  $P(X)=a(X(x_0))$). In that case, we have
$h(\varphi,P)=\frac{1}{2}a\cdot a$, hence $\partial_\varphi h(\varphi,P)=0$ and therefore $P$ is constant along the geodesic flow, and
$
\partial_t\varphi(t,x)=e^{-\frac{\vert \varphi(t,x)-\varphi(t,x_0)\vert^2}{2\sigma}}a^T.
$
Note that, in this case, the particle $\varphi(t,x_0)=x_0+t a^T$ is a straight line with constant speed $a\cdot a$.

\subsubsection{Gaussian kernels for sub-Riemannian distributions in $\R^d$}
Let $X_1,\dots,X_k$ be smooth pointwise linearly independent vector fields on $\R^d$, with $k\leq d$, which are bounded as well as all their derivatives. These vector fields generate a sub-Riemannian structure on $\R^d$ (see \cite{Bellaiche,MBOOK} for this classical construction), with horizontal curves $t\mapsto x(t)$ satisfying the differential equation $\dot{x}(t)=\sum_{j=1}^ku_j(t)X_j(x(t))$ for almost every $t$, with $u_j\in L^2(0,1;\R)$ for every $j=1,\ldots,k$.
We consider the kernel $K(x,y)p=e(x,y)\sum_{j=1}^k p(X_j(y)) X_j(x)$.

Let us first prove that $K$ is the reproducing kernel of a Hilbert space of vector fields $\mathcal{H}_e$. For any compactly supported one-form $P$ with distributional coefficients, an easy computation gives
$K(x,\cdot)P=\sum_{j=1}^k\left(\int_{\R^d}e(x-y)P(y)(X_j(y))\,\d y\right)X_j(x)$,
where $P(y)(X_j(y))= P_1(y)X^1_j(y)+\cdots+P_d(y)X^d_j(y)$. The vector $K(x,\cdot)P$ is well defined since the vector fields $y\mapsto e(x-y)X_j(y)$ are smooth and all their derivatives decrease exponentially at infinity, and $P$ is a compactly supported one-form with distributional coefficients.
Then we have
$$
P(K(\cdot,\cdot)P)=\sum_{j=1}^k\int_{\R^d}e(x-y)P(y)(X_j(y))P(x)(X_j(x))\,\d y\,\d x.
$$
In order to check that $K$ is the reproducing kernel of a Hilbert space of vector fields, it is sufficient to check that $P(K(\cdot,\cdot)P)\geq 0$ for any such one-form $P$, and $P(K(\cdot,\cdot)P)=0$ if and only if the vector field $x\mapsto K(x,\cdot)P$ is identically equal to $0$ (see \cite{YBOOK}).
For such a $P$, we set $p_j(x)=P(x)(X_j(x))$, for $j=1,\dots,k$.
Then we have
$K(x,\cdot)P=\sum_{j=1}^k\int_{\R^d}e(x-y)p_j(y)\,\d y \ X_j(x)$,
and therefore
$
P(K(\cdot,\cdot)P)=\sum_{j=1}^k\int_{\R^d}e(x-y)p_j(y)p_j(x)\,\d y\,\d x.
$
Since $(x,y)\mapsto e(x-y)$ is the heat kernel at time $\sigma$, it follows that $\int_{\R^d}e(x-y)T(y)T(x)\,\d y\,\d x \geq 0$ for any distribution $T$ on $\R^d$, with equality if and only if $T=0$.
Therefore, $P(K(\cdot,\cdot)P)$ is nonnegative, and is equal to $0$ if and only if $p_1=\dots=p_k=0$, in which case $K(x,\cdot)P=0$ for every $x$ in $\R^d$. We have thus proved that $K$ is the reproducing kernel of a Hilbert space of vector fields $\mathcal{H}_e$. 

Moreover, any element of $\mathcal{H}_e$ can be written as $x\mapsto u_1(x)X_1(x)+\cdots+u_k(x)X_k(x)$ (and is horizontal with respect to the sub-Riemannian structure on $\R^d$ induced by the $X_i$), where the $u_j$'s are analytic functions with all derivatives decreasing exponentially at infinity. In particular, the inclusions $\mathcal{H}_e\subset H^s(\R^d,\R^d)$ are  continuous for every integer $s$.

\begin{remark}
It could be more natural to replace $e(x-y)$ with the sub-Riemannian heat kernel associated with the sub-Laplacian $X_1^2+\dots+X_k^2$, but since such kernels are much harder to compute, we will keep the Euclidean heat kernel.
\end{remark}

For every $(\varphi,P)\in T^*\mathcal{D}^{s}(\R^d)= \mathcal{D}^{s}(\R^d)\times H^{-s}(\R^d,\R^d)$, the solution of $\partial_XH(\varphi,P,X)=0$ is
$X(\varphi,P)(x) = K(x,\varphi(\cdot))P =\sum_{j=1}^k\left(\int_{\R^d}e(x-\varphi(y))P(y)(X_j(\varphi(y)))\,\d y \right)X_j(x)$.
The normal Hamiltonian is given by
$$
h(\varphi,P)
=
\frac{1}{2}\sum_{j=1}^k \int_{\R^d\times\R^d}e(\varphi(x)-\varphi(y))P(y)(X_j(\varphi(y)))P(x)(X_j(\varphi(x)))\,\d y\,\d x .
$$
We infer that the normal geodesic equations are written as
$$
\begin{aligned}
\partial_t{\varphi}(t,x)&=\sum_{j=1}^k\left(\int_{\R^d} e(\varphi(t,x)-\varphi(t,y))P(t,y)(X_j(\varphi(t,y)))\,\d y\right)X_j(\varphi(t,x)),\\
\partial_t{P}(t,x)&=\frac{1}{\sigma}\sum_{j=1}^k\int_{\R^d}e(\varphi(t,x)-\varphi(t,y))(\varphi(t,x)-\varphi(t,y))^T P(t,y)(X_j(\varphi(t,y)))P(t,x)(X_j(\varphi(t,x))) \,\d y\\
&
\ \ \ \,
-
\sum_{j=1}^k\int_{\R^d}e(\varphi(t,x)-\varphi(t,y))P(t,y)(X_j(\varphi(t,y)))
P(x)(\d X_j(\varphi(t,x)))\,\d y.
\end{aligned}
$$
Here, we have $P(x)(\d X_j(\varphi(t,x))) = P_1(x)\d X^1_j(\varphi(t,x)) +\dots+P_d(x)\d X^d_j(\varphi(t,x))$.
It is easy to check that this geodesic flow preserves the support of $P$ and its regularity. 

In the simple case where $P=a\otimes\delta_{x_0}$ with $a\in(\R^d)^*$, we have
$$
X(\varphi,P)(x)=e(x-\varphi(x_0))\sum_{j=1}^ka(X_j(\varphi(x_0)))X_j(x).
$$
Since $e(0)=1,$ we get
$h(\varphi,a\otimes\delta_{x_0})=\frac{1}{2}\sum_{j=1}^ka(X_j(\varphi(x_0)))^2$.
It is interesting to note that $h(\varphi,a\otimes\delta_{x_0})=h^{\Delta}(\varphi(x_0),a)$, where $h^{\Delta}$ is the normal hamiltonian for the sub-Riemannian structure induced by the $X_j$'s on $\R^d$. 
Now we have
$
\partial_\varphi h(\varphi,a\otimes\delta_{x_0})=
\left(\sum_{r=1}^ka(X_r(\varphi(x_0)))a(\d X_{r,\varphi(x_0)})\right)\otimes\delta_{x_0}.
$
We see that the subbundle $\mathcal{D}^s(\R^d)\times(\R^{d*}\otimes\delta_{x_0})\subset T^*\mathcal{D}^s(\R^d)$ is invariant under the Hamiltonian geodesic flow. Hence, if $t\mapsto(\varphi(t),P(t))$ is solution of the normal geodesic equations and $P(0)=a(0)\otimes\delta_{x_0}$, then there exists a curve $t\mapsto a(t)\in\R^d$ such that $P(t)=a(t)\otimes\delta_{x_0}$ for every $t$.
Denoting by $x(t)=\varphi(t,x)$ for $x\in\R^d$ and by $x_0(t)=\varphi(t,x_0)$, we finally obtain
$$
\dot{x}(t)= e(x(t)-x_0(t))\sum_{j=1}^ka(t)(X_i(x_0(t)))X_j(x(t)),\quad
\dot{a}(t)=-\sum_{j=1}^ka(X_r(x_0(t)))a(\d X_j(x_0(t))).
$$
In particular, if $x(0)=x_0$, we get
$$
\dot{x}(t)=\sum_{j=1}^ka(X_j(x(t)))X_j(x(t)),\quad
\dot{a}(t)=-\sum_{j=1}^ka(X_j(x(t)))a(\d X_j(x(t))).
$$
We recover the equations satisfied by a normal geodesic in $\R^d$ for the sub-Riemannian structure induced by the $X_j$'s on $\R^d$, with initial covector $a(0)$.

\subsection{Singular curves with Dirac momenta in shape spaces of landmarks}\label{singdirac}
Let $M$ be a complete Riemannian manifold with bounded geometry, of dimension $d$, let $s>d/2+1$ be an integer, and let $\mathcal{H}_e$ be a Hilbert space with continuous inclusion in $\Gamma^{s+1}(TM)$.

It is difficult to give a complete description of any possible singular curve for the sub-Riemannian structure induced by $\mathcal{H}_e$. However, we know from Theorem \ref{thm_abnormal} that any such curve $\varphi(\cdot)$ is associated with a singular covector $t\mapsto P(t)\in T^*_{\varphi(t)}\mathcal{D}^s(M)$, satisfying \eqref{abgeod1} and \eqref{abgeod2}. 
In this section, we focus on those singular curves whose singular covector is a finite sum of Dirac masses, and we prove that they correspond to abnormal curves for a certain finite-dimensional control system in certain manifolds, called shape spaces of landmarks. 

Let us first explain what a shape space of landmarks is.

\begin{definition}
For every integer $n\in \N^*$, the manifold of $n$ landmarks of $M$ is defined by
$$
\mathrm{Lmk}^n(M)=\{(x_1,\dots,x_n)\in M^n\mid i\neq j\Rightarrow\, x_i\neq x_j\}.
$$
\end{definition}

Landmarks manifolds are of great interest in shape analysis (see \cite{YBOOK}). Since $s>d/2+1$, the group $\mathcal{D}^s(M)$ has a $\mathcal{C}^1$ action on the manifold $\mathrm{Lmk}^n(M)$, defined by
$$
\varphi\cdot (x_1,\dots,x_n)=(\varphi(x_1),\dots,\varphi(x_n)).
$$
Note that, by definition of $\mathrm{Lmk}^n(M)$, the mapping $R_q:\varphi\mapsto\varphi\cdot q$ is a submersion, for every $q=(x_1,\dots,x_n)$.

To this differentiable action is associated the infinitesimal action $\xi$, defined as the linear bundle morphism $\Gamma^s(TM)\times \mathrm{Lmk}^n(M)\rightarrow T\mathrm{Lmk}^n(M)$ given by
$$
\xi_qX=(X(x_1),\dots,X(x_n)), 
$$
for every $q=(x_1,\dots,x_n)\in \mathrm{Lmk}^n(M)$ and every $X\in \Gamma^s(TM)$.
Then $\mathrm{Lmk}^n(M)$ turns out to be a \emph{shape space} (as defined in \cite{ATY}). Restricting this morphism to $\mathcal{H}_e\times \mathrm{Lmk}^n(M)$, we obtain the control system 
\begin{equation}\label{cslmk}
\dot{q}(t)=\xi_{q(t)}X(t), 
\end{equation}
for almost every $t\in[0,1]$, that is, $\dot{x}_i(t)=X(t,x_i(t))$, $i=1,\dots,n$. 
In this control system, the control is $X(\cdot)\in L^2(0,1;\mathcal{H}_e)$.
Note that, if $\varphi(t)$ is the flow of $X$, defined as the unique solution of the Cauchy problem $\dot{\varphi}(t)=X(t)\circ\varphi(t)$ and $\varphi(0)=\mathrm{Id}_M$, then $q(t)=\varphi(t)\cdot q(0)$ for every $t$.

This control system can be as well seen as a rank-varying sub-Riemannian structure on $\mathrm{Lmk}^n(M)$ (as defined in \cite{ABCG}).

For every $q_0\in \mathrm{Lmk}^n(M)$, the end-point mapping $\mathrm{end}^n_{q_0}:L^2(0,1;\mathcal{H}_e)\rightarrow \mathrm{Lmk}^n(M)$ is defined by $\mathrm{end}^n_{q_0}(X(\cdot))=q(1)$, where $q(\cdot)$ is the unique solution of the Cauchy problem $\dot{q}(t)=\xi_{q(t)}X(t)$, $q(0)=q_0$.
Obviously, we have $\mathrm{end}^n_{q_0}=R_{q_0}\circ\mathrm{end}_e$, where $R_{q_0}(\varphi)=\varphi\cdot q_0$ (and where $\mathrm{end}_e$ is the end-point mapping defined in Definition \ref{end-point}).


\begin{proposition}
Let $\varphi(\cdot)\in H^1(0,1;\mathcal{D}^s(M))$ be an horizontal curve starting at $e$, with logarithmic velocity $X(\cdot)$. Let $q_0=(x_1,\dots,x_n)\in \mathrm{Lmk}^n(M)$, and let $q(\cdot)=\varphi(\cdot)\cdot q_0$ be the corresponding curve on $\mathrm{Lmk}^n(M)$. 
The curve $\varphi(\cdot)$ is abnormal, associated with a singular covector $P(\cdot)$ such that 
$$P(1)=\sum_{i=1}^n p_{i}(1)\otimes \delta_{x_{i}},\quad p_{i}(1)\in T^*_{x_{i,0}}M,$$
if and only if $q(\cdot)$ is an abnormal curve for the control system \eqref{cslmk}.
\end{proposition}

\begin{proof}
Since $q(t)=\varphi(t)\cdot q(0)$, we have $\mathrm{end}^n_{q_0}(X(\cdot))=\mathrm{end}_{e}(X(\cdot))\cdot q_0$. In the finite-dimensional manifold $\mathrm{Lmk}^n(M)$, the curve $q(\cdot)$ is abnormal if and only if (see, e.g., \cite{ChitourJeanTrelatJDG}) there exists $p(1)=(p_1(1),\dots,p_n(1))\in T^*_{q(1)}\mathrm{Lmk}^n(M)\setminus\{0\}$ such that $d\,\mathrm{end}^n_{q_0}(X(\cdot))^*.p(1)=0$.
Since the mapping $\varphi\mapsto \varphi\cdot q_0$ is a submersion and $\mathrm{end}^n_{q_0}=R_{q_0}\circ\mathrm{end}_e$, this is equivalent to the existence of $P(1)\in T_{\varphi(1)}^*\mathcal{D}^s(M)\setminus\{0\}$ such that $d\,\mathrm{end}_e(X)^* .P(1)=0$ and such that $\ker \d R_{q_0}(\varphi(1))\subset\ker(P(1))$, which means, since $\d R_{q_0}(\varphi(1)).\delta \varphi=(\delta\varphi(x_1),\dots,\delta\varphi(x_n))$, that $P(1)=p_1(1)\otimes \delta_{x_1}+\dots+p_n(1)\otimes\delta_{x_n}$.
\end{proof}

\begin{remark}
Since the support of $P(\cdot)$ and its order as a distribution remain constant in time, we infer that, for every $t\in [0,1]$, there exist $p_i(t)\in T^*_{\varphi(t,x_i)}M,$ $i=1,\dots,n$ such that $P(t)=p_1(t)\otimes \delta_{x_1}+\dots+p_n(t)\otimes\delta_{x_n}$. Using the momentum formulation from Section \ref{momentum}, and setting $p(t)=(p_1(t),\dots,p_n(t))$, it is easy to check that $(q(\cdot),p(\cdot))$ is a curve on $T^*\mathrm{Lmk}^n(M)$, satisfying the abnormal Hamiltonian equations associated with the control system \eqref{cslmk} (see \cite{ChitourJeanTrelatJDG} for a detailed analysis of those equations).
\end{remark}

\begin{remark}
Applying the above results with $n=1$, we get the statement claimed in Remark \ref{rem_sing2}: singular curves on $M$ induce singular curves on $\mathcal{D}^s(M)$. This resonates strongly with the next section, where we will see that reachability in $M$ implies reachability in $\mathcal{D}^s(M)$ (at least when $k=0$).
\end{remark}

\section{Reachability properties in the group of diffeomorphims}\label{sec3}
Throughout this section, $(\mathcal{H}_e,\langle\cdot,\cdot\rangle)$ is a Hilbert space of vector fields of class at least $H^s$ on a Riemannian manifold $M$ of bounded geometry and of dimension $d$, with continuous inclusion in $\Gamma^s(TM)$ and $s$ is an integer such that $s>d/2+1$.
According to Definition \ref{srstruc}, we consider the right-invariant sub-Riemannian structure induced by $\mathcal{H}_e$ on $\mathcal{D}^s(M)$.

The purpose of this section is to provide sufficient conditions on $\mathcal{H}_e$ ensuring approximate or exact reachability from $e$.

\begin{definition}
The \emph{reachable set} from $e=\mathrm{id}_M$ is defined by
$$
\mathcal{R}(e) = \{\varphi\in\mathcal{D}^s(M) \mid d_{SR}(e,\varphi)<+\infty\} .
$$
In other words, $\mathcal{R}(e)$ is the set of all $\varphi\in\mathcal{D}^s(M)$ that are in the image of $\mathrm{end}_{e}$, i.e., that can be connected from $e$ by means of an horizontal curve $\varphi(\cdot)\in H^1(0,1;\mathcal{D}^s(M))$.

We say that $\varphi\in\mathcal{D}^s(M)$ is reachable from $e$ if $\varphi\in\mathcal{R}(e)$, and is approximately reachable from $e$ if $\varphi$ belongs to the closure of $\mathcal{R}(e)$ in $\mathcal{D}^s(M)$.
\end{definition}

Hereafter, we first establish a general approximate reachability result. However, in such a general context, we cannot hope to have stronger reachability properties. In the more particular case where the sub-Riemannian structure on $\mathcal{D}^s(M)$ is coming from a finite-dimensional structure (as in Example \ref{subriem}), we establish an exact reachability property.

\subsection{Approximate reachability}
We start with the following simple result.

\begin{proposition}\label{prop_dense}
If $\mathcal{H}_e$ is dense in $\Gamma^s(TM)$, then $\mathcal{R}(e)$ is dense in $\mathcal{D}^s(M)$.
\end{proposition}

This result holds true as well for $s=+\infty$. Note that, for $s=+\infty$, it was generalized to the context of convenient spaces in \cite{KM}.

It can also be noted that it is not required to assume that $M$ is connected. What is important is that $\mathcal{D}^s(M)$ itself is connected (by definition).

\begin{proof}
The assumption implies that $\mathcal{H}^s$ is dense in $T\mathcal{D}^s(M)$. 
We recall the following general lemma.

\begin{lemma}
Let $\mathcal{M}$ be a connected Banach manifold, and let $B\subset\mathcal{M}$ be a closed subset such that, for every $q\in B,$ the set of initial velocities of curves in $B$ starting at $q$ is dense in $T_q\mathcal{M}$. Then $B=\mathcal{M}$.
\end{lemma}

This lemma has been proved in \cite{DS}, and rediscovered in \cite[Theorem D and Corollary A.2]{HeintzeLiu_AM1999} (where the proof is more readable).
Since the closure of an orbit is an union of orbits, the result follows.
\end{proof}

A much more general result can be inferred from \cite{AC} in the case where $M$ is compact.

\begin{proposition}\label{propAC}
We assume that $M$ is compact, and that
$$
\left\lbrace \sum_{i=1}^r u_i X_i\mid u_1,\dots,u_r\in \mathcal{C}^\infty(M) \right\rbrace\subset\mathcal{H}_e,
$$
where $X_1,\dots,X_r$ are smooth vector fields on $M$, such that any two points $x$ and $y$ of $M$ can be connected by a smooth curve $x(\cdot)$ on $M$ whose velocity belongs to $\Delta = \mathrm{Span}\{ X_iÊ\mid i=1,\dots,r\}$ almost everywhere.
Then $\mathcal{R}(e)$ contains the set $\mathcal{D}^\infty(M)$ of all smooth diffeomorphisms of $\mathcal{D}^s(M)$. In particular, it is dense in $\mathcal{D}^s(M)$.
\end{proposition}

\begin{proof}
The main result of \cite{AC} (which is actually slightly stronger) states that, if $M$ is compact and if any two points $x,y\in M$ can be connected by a smooth curve $x(\cdot)$ on $M$ whose velocity belongs to $\Delta$ almost everywhere, then there exists an integer $m\in\N$ such that, for every $\varphi\in \mathcal{D}^\infty(M)$, there exist functions $u_1,\dots,u_m$ on $M$, of class $\mathcal{C}^\infty$, and integers $i_1,\dots,i_m$ in the set $\{1,\dots,r\}$, such that $\varphi=\varphi^{u_1X_{i_1}}(1)\circ\dots\circ\varphi^{u_mX_{i_m}}(1)$.
The result follows.
\end{proof}

Proposition \ref{propAC} says that an exact reachability property for a given smooth sub-Riemannian manifold $(M,\Delta,g)$ with $M$ compact implies an approximate reachability property on $\mathcal{D}^s(M)$ endowed with the strong right-invariant sub-Riemannian structure induced by $\mathcal{H}_e$, provided that $\Delta\subset \mathcal{H}_e$.

This proposition can be applied in the framework of Example \ref{subriem}. A well known sufficient condition (which is necessary in the analytic case) on a connected manifold $M$ ensuring that any two points of $M$ can be joined by an horizontal curve is that the Lie algebra generated by the vector fields $X_1,\ldots,X_r$ coincides with the whole tangent space $TM$ (bracket-generating assumption). Under this slightly stronger assumption, we actually have an exact reachability result (see next section).

\subsection{Exact reachability}
Establishing exact reachability (i.e., $\mathcal{R}(e)=\mathcal{D}^s(M)$) is hopeless for general infinite-dimensional sub-Riemannian manifolds, unless one has specific assumptions. However, the proof of \cite{AC} can easily be generalized to the $H^s$ case, in the following context. 

Throughout the section, we make the additional assumption that the manifold $M$ is compact.

Let $r\in\N^*$, and let $X_1,\dots, X_r$ be smooth vector fields on $M$. The family $(X_1,\dots, X_r)$ induces a (possibly rank-varying) sub-Riemannian structure on $M$ (see \cite{Bellaiche}). Note that $\Delta = \mathrm{Span} \{ X_1,\dots, X_r \}$ being a rank-varying subbundle of $TM$ does not raise any problem (see \cite{ABCG} where sub-Riemannian manifolds are defined in a more general way). An horizontal curve $x(\cdot)$ on $M$ (also called $\Delta$-horizontal curve) is a curve whose velocity belongs to $\Delta$ almost everywhere.

\begin{theorem}\label{consagrachev}
We assume that 
$$
\left\lbrace \sum_{i=1}^r u_i X_i\mid u_1,\dots,u_r\in H^s(M) \right\rbrace\subset\mathcal{H}_e,
$$
and that any two points $x$ and $y$ belonging to the same connected component of $M$ can be joined by a smooth horizontal curve $x(\cdot)$ on $M$.
Then $\mathcal{R}(e)=\mathcal{D}^s(M)$.
\end{theorem}

Note that we do not need to assume that $M$ is connected in this result.

\begin{remark} 
Since the proof provided in \cite{AC} is compatible with the ILH structure of $\mathcal{D}^\infty(M)$ (inverse limit of $\mathcal{D}^s(M)$ as $s\rightarrow+\infty$), we infer that the above exact reachability is true as well on $\mathcal{D}^\infty(M)$ when taking the $u_i$'s in $\mathcal{C}^\infty(M)$. In particular, we recover Proposition \ref{propAC}.
\end{remark}

We not not provide a proof of that result, which is a straightforward extension of the main result of \cite{AC}. 

Actually, under the slightly stronger assumption that the vector fields $X_1,\dots,X_r$ are bracket-generating, we derive hereafter a stronger result, establishing not only that $\mathcal{R}(e)=\mathcal{D}^s(M)$, but also that the topology induced by the sub-Riemannian distance coincides with the intrinsic manifold topology of $\mathcal{D}^s(M)$; and this, with a proof that is much simpler and shorter than the one of \cite{AC}.

The family $(X_1,\dots,X_r)$ is said to be bracket-generating if
$$
T_xM=\mathrm{Span}\left\lbrace [X_{i_1}[\dots,[X_{i_{j-1}},X_{i_j}]\dots](x) \mid   j\in\N^*,\ 1\leq i_1,\dots,i_j\leq r   \right\rbrace,
$$
for every $x\in M$. This means that any tangent vector at $x\in M$ is a linear combination of iterated Lie brackets of $X_1,\dots,X_k$, in other words, if $TM=\mathrm{Lie}(\Delta)$.
Under this assumption, {and assuming that $M$ is connected,} the well-known Chow-Rashevski theorem (see \cite{Bellaiche,MBOOK}) states that any two points of $M$ belonging to the same connected component of $M$ can be joined by an horizontal curve for the sub-Riemannian structure on $M$.
\footnote{The converse is only true for analytic vector fields: if $M$ and the vector fields $X_i$'s are analytic, and if any two points of $M$ can be connected by a horizontal curve, then the $X_i$'s are bracket-generating.}

For any given family of indices $I=(i_1,\dots,i_j)$ with $ 1\leq i_1,\dots,i_j\leq r$, we denote
\begin{equation}\label{cro}
X_I=[X_{i_1}[\dots,[X_{i_{j-1}},X_{i_j}]\dots].
\end{equation}
{Since} $M$ is compact, the family $(X_1,\dots,X_r)$ is bracket-generating if and only if there exists a fixed family of indices $I_1,\dots,I_m$, with $I_l=(i_1^l,\dots,i_{j^l}^l)\subset \{1,\dots,r\}^{j^l}$, such that $T_xM=\mathrm{Span}(X_{I_1}(x),\dots,X_{I_m}(x))$ for every $x\in M$. Now let
\begin{equation}\label{defHeSR}
\mathcal{H}_e=\{X\in\Gamma^s(TM)\mid\forall x\in M \quad X(x)\in \mathrm{Span}(X_1(x),\dots,X_r(x))\},
\end{equation}
on which we define a Hilbert product $\left\langle\cdot,\cdot\right\rangle$ whose norm is equivalent to the $H^s$ norm \eqref{sobnorm} (see Section \ref{sec_defi}). Note that $\mathcal{H}_e$ is an $H^s(M)$-module generated by $X_1,\dots, X_r$, so that any $X\in \mathcal{H}_e$ can be written as
$$
X=\sum_{i=1}^ru^iX_i,\quad u=(u^1,\dots,u^r)\in H^s(M,\R^r),
$$
and we have
$$
\left\langle X,X\right\rangle\leq \sum_{i=1}^r\Vert u^i\Vert_{H^s(M)}^2.
$$
We consider the strong right-invariant sub-Riemannian structure on $\mathcal{D}^s(M)$ induced by the Hilbert space $(\mathcal{H}_e,\left\langle\cdot,\cdot\right\rangle)$ defined by \eqref{defHeSR}. The corresponding sub-Riemannian distance is denoted by $d_{SR}$. Note that
$$
\mathcal{H}_\varphi=\left\lbrace\sum_{i=1}^r (u^i\circ\varphi) \, X_i\circ\varphi\mid u=(u^1,\dots,u^r)\in H^s(M,\R^r) \right\rbrace,
$$
for every $\varphi\in \mathcal{D}^s(M)$.

\begin{remark}
Under the assumption that the family $(X_1,\ldots,X_r$) is bracket-generating, we have
$$
\Gamma^s(TM)=\left\lbrace\sum_{i=1}^m u^iX_{I_i} \mid u=(u^1,\dots,u^m)\in H^s(M,\R^m)\right\rbrace.
$$
\end{remark}

\begin{theorem}\label{ball}
We assume that 
the family $(X_1,\dots,X_r)$ is bracket-generating. Then there exist $C>0$, a neighborhood $\mathcal{U}$ of $0$ in $H^{s}(M,\R^m)$, and a $\mathcal{C}^1$-submersion $\phi:\mathcal{U}\rightarrow \mathcal{D}^s(M)$, with $\phi(0)=e$, such that
\begin{equation}
\label{inegballbox}
d_{SR}(e,\phi(u_1,\dots,u_m))\leq C\sum_{i=1}^m \Vert u_i\Vert_s^{1/j^i}.
\end{equation}
As a consequence, we have $\mathcal{R}(e)=\mathcal{D}^s(M)$, and the topology induced on $\mathcal{D}^s(M)$ by the sub-Riemannian distance $d_{SR}$ coincides with the intrinsic manifold topology of $\mathcal{D}^s(M)$. 
\end{theorem}

Theorem \ref{ball} is proved in Appendix \ref{appendix_proofball}.
%
%

\begin{remark} 
Note that, as discussed in Remark \ref{rem_sing}, it is necessary to assume that $k=0$ in order to obtain exact reachability (for $k\geq 1$ we {\em never} have exact reachability).
But the fact that $k=0$ causes some difficulties in the proof, in particular because the end-point mapping is then only continuous. To overcome this problem, in the proof we use an equivalent sub-Riemannian structure which is smooth but not right-invariant.
\end{remark}

\begin{remark}
As it easily follows from our proof, $\phi$ is actually an ILH (``Inverse Limit Hilbert") submersion of class $\mathcal{C}^1$. This means that its restriction to $H^{s+k}(M,\R^m)$ is a submersion onto $\mathcal{D}^{s+k}(M)$ at $e$, for every $k\in\N$. In particular, this fact remains true when restricted to the inverse limits $\mathcal{C}^\infty(M,\R^m)$ and $\mathcal{D}^\infty(M)$, and we recover the main result of \cite{AC} for bracket-generating distributions.
Although the exact reachability with controls in $H^s$ is a straightforward generalization of \cite{AC} (as already said), the result concerning the induced topology is new and cannot be deduced from the proof of \cite{AC}. Theorem \ref{ball} also provides a generalization of (half of) the ball-box theorem to infinite-dimensional sub-Riemannian geometry (see, e.g., \cite{MBOOK} for the classical ball-box theorem in finite dimension).
Establishing the converse inequality (``second half" of the ball-box theorem) is an open problem in infinite dimension, and does not seem to be straightforward. Indeed, the classical proof in finite dimension uses the concept of privileged coordinates (\cite{Bellaiche,MBOOK}), which seems hard to generalize to our case.
\end{remark}

\begin{remark}\label{weaktopo}
The proof of Theorem \ref{ball} can easily be generalized to weak and non-right-invariant metrics (in which case the sub-Riemannian topology is coarser than the intrinsic one). As a consequence, we expect that estimates similar to \eqref{inegballbox} can be established, when restricting the weak Riemannian metric (which is \textit{not} right-invariant) from \cite{EM} to $\mathcal{H}^s$, that is,
$$
d_{SR}(e,\phi(u_1,\dots,u_m))\leq C\sum_{i=1}^m \Vert u_i\Vert_{L^2}^{1/j^i}.
$$
\end{remark}

\begin{remark}
We stress again that, in the present section, we have taken a specific definition of $\mathcal{H}_e$. In more general cases, only a dense subset can be reached from a given point using horizontal curves, and the topologies do not coincide. Indeed, the bracket-generating condition is an algebraic condition, and therefore, in infinite dimension, the space generated by linear combinations of brackets of horizontal vector fields is only dense in the tangent space of the manifold. 
\end{remark}

\begin{remark}
Theorem \ref{ball} generalizes some results established in \cite{BV} in the ``Riemannian case" $\mathcal{H}_e=\Gamma^s(TM)$.
\end{remark}

\begin{remark}
A curve $\varphi(\cdot)\in H^1(0,1;\mathcal{D}^s(M))$ is horizontal for the strong right-invariant sub-Riemannian structure on $\mathcal{D}^s(M)$ induced by the Hilbert space $(\mathcal{H}_e,\left\langle\cdot,\cdot\right\rangle)$ defined by \eqref{defHeSR}, if and only if, for every $x\in M$, the curve $t\mapsto \varphi(t,x)\in M$ is horizontal for the (finite-dimensional) sub-Riemannian structure generated on $M$ by the family $(X_1,\dots,X_r)$.

It is then interesting to provide an interpretation in terms of classical transport equations and of particle flow.
Denoting by $\mathrm{div}$ the divergence operator associated with the canonical Riemannian volume of $M$, it is well known by the DiPerna-Lions theory (see \cite{Ambrosio,DiPernaLions}) that, for any given time-dependent Lipschitz vector field $X(\cdot):[0,1]\rightarrow\Gamma(TM)$ with time-integrable Lipschitz coefficients, generating a flow $\varphi(\cdot)$ on $M$, for every $\mu_0\in \mathcal{P}(M)$, the transport equation
$$
\partial_t\mu(t)+\mathrm{div}(\mu(t) X(t))=0,
$$
has a unique (measure) solution in $C^0(\R,\mathcal{P}(M))$ such that $\mu(0)=\mu_0$, given by $\mu(t)=\varphi(t)_*\mu_0$, which is the image measure (pushforward) of $\mu_0$ under $\varphi(t)$. Here, $\mathcal{P}(M)$ is the set of probability measures on $M$. In this context, for every $x\in M$, the curve $t\mapsto \varphi(t,x)$ is usually called a \textit{particle} (starting at $x$), and the flow $\varphi(\cdot)$ is often referred to as the \textit{particle flow}.

If one considers time-dependent vector fields $X(\cdot)$ that are, for almost every time, a linear combination (with time-integrable Lipschitz coefficients) of the vector fields $X_1,\ldots,X_r$, then the particle flow $\varphi(\cdot)$ is horizontal for the sub-Riemannian structure on $\mathcal{D}^s(M)$ considered above, and the particles are exactly the horizontal curves of $M$.

In this context, it is then easy to derive, as a consequence of Theorem \ref{ball}, a controllability theorem for controlled transport PDE's, of the form
$$\partial_t\mu(t)+\sum_{i=1}^r \mathrm{div}(\mu(t) u_i(t)X_i)=0,\quad u_i(t)\in H^s(M).$$
More precisely, one has controllability in the space of absolutely continuous probability measures. But this fact is exactly equivalent to a version of the famous Moser theorem on volume forms, that we are going to explore in more details.
Hence, this remark makes the transition with further considerations on what can be done with  horizontal flows of diffeomorphisms. In the next section we are going to revisit the Moser trick in the context of sub-Riemannian geometry.
\end{remark}

\subsection{Moser theorems with horizontal flows}\label{sec_Moser}
In this section, we provide some applications to the ``horizontal" transport of symplectic forms and of volume forms on a compact manifold. 

As in the previous section, we assume that the manifold $M$ is Riemannian and compact. The canonical Riemannian measure is denoted by $dx_g$.
Let $r\in\N^*$, and let $X_1,\dots, X_r$ be smooth vector fields on $M$. We denote, as previously, $\Delta = \mathrm{Span} \{ X_1,\dots, X_r \}$.
We consider on $\mathcal{D}^s(M)$ the strong right-invariant sub-Riemannian structure induced by the Hilbert space $(\mathcal{H}_e,\left\langle\cdot,\cdot\right\rangle)$ defined by \eqref{defHeSR}.

\begin{theorem}\label{thmMoser}
We assume that $s>d/2+1$, and that any two points of $M$ can be joined by a $\Delta$-horizontal curve.
\begin{enumerate}
\item Let $\mu_0$ and $\mu_1$ be two volume forms with respective densities $f_0$ and $f_1$ of class $H^{s-1}$ on $M$. If $\int_Mf_0(x)\, \d x_g=\int_Mf_1(x)\, \d x_g$, then there exists an horizontal curve $\varphi(\cdot)\in H^1(0,1;\mathcal{D}^{s}(M))$ such that $\varphi(0)=e$ and $\varphi(1)_*\mu_0=\mu_1$.
\item Let $\omega_0$ and $\omega_1$ be two symplectic forms on $M$, with coefficients of class $H^{s-1}$. If $\omega_0$ and $\omega_1$ belong to the same connected component of the same cohomology class, then there exists an horizontal curve $\varphi(\cdot)\in H^1(0,1;\mathcal{D}^{s}(M))$ such that $\varphi(0)=e$ and $\varphi(1)_*\omega_0=\omega_1$.
\end{enumerate}
This result remains true when $s=+\infty$.
\end{theorem}

Theorem \ref{thmMoser} is a nonholonomic version of the usual well known Moser theorem (see \cite{Moser}). The first part has been proved in \cite{KhesinLee} for $s=+\infty$ (and the second part was conjectured in that reference).
The difference with the usual statement is twofold: first, our statement here is for diffeomorphisms of class $H^s$, for any $s$ large enough; second, in the Moser trick, we show here that the path of diffeomorphisms joining the initial (volume or symplectic) form to the target one can be chosen to be horizontal.

\begin{proof}
Let $\mu_0$ and $\mu_1$ be two volume forms of class $H^{s-1}$ having the same total volume $\alpha\geq 0$ (the proof works exactly in the same way for symplectic forms).
We denote by $\mathrm{Vol}^{s-1}_\alpha(M)$ the (convex) set of volume forms $f\, dx_g$ on $M$ such that $f\in H^{s-1}(M)$ and $\int_Mf\, dx_g=\alpha$.

The proof goes in two steps.

First, we use the standard Moser trick, without taking care of the horizontal condition. Let us recall this very classical method.
Let $Y(\cdot)=\dot\varphi(\cdot)\circ\varphi(\cdot)^{-1}\in L^2(0,1;\Gamma^{s}(M))$ be arbitrary, and let $\psi(\cdot)\in H^1(0,1;\mathcal{D}^{s}(M))$ be the unique solution of $\dot\psi(\cdot)=Y(\cdot)\circ\psi(\cdot)$ such that $\psi(0)=e$.
Let $\mu(\cdot)\in H^1(0,1;\mathrm{Vol}^{s-1}_\alpha(M))$ be an arbitrary path of volume forms. Using the Lie derivative of a time-dependent vector field, we have
$$\frac{d}{dt} \psi(t)^*\mu(t) = \psi(t)^* \left( L_{Y(t)}\mu(t) + \dot\mu(t)  \right),$$
and therefore, if $L_{Y(t)}\mu(t) + \dot\mu(t)=0$ for almost every $t\in[0,1]$ then $\mu(0)=\psi(1)^*\mu(1)$. 
We choose the linear path $\mu(t)=(1-t)\mu_0+t\mu_1$. Let us then search a time-dependent vector field $Y(\cdot)$ such that $L_{Y(t)}\mu(t)=-\dot\mu(t)=\mu_0-\mu_1$, for any fixed time $t$. Using the Cartan formula $L_Y=d\iota_Y+\iota_Y d$, the fact that $d\mu(t)=0$, and the fact that $\mu_0-\mu_1=d\eta$ for some $(n-1)$-form $\eta$ having coefficients of class $H^{s}$, it suffices to solve $\iota_{Y(t)}\mu(t)=\eta$, which has a solution $Y(t)$ of class $H^{s}$ because $\mu(t)$ is non-degenerate. The time-dependent vector field $Y(\cdot)$ generates a flow $\psi(\cdot)\in H^1(0,1;\mathcal{D}^{s}(M))$ such that $\psi(0)=e$, and with the above calculation we have $\mu_0=\psi(1)^*\mu_1$ and thus $\mu_1=\psi(1)_*\mu_0$.

Now, using Theorem \ref{ball}, there exists an horizontal curve $\varphi(\cdot)\in H^1(0,1;\mathcal{D}^s(M))$ such that $\varphi(0)=e$ and $\varphi(1)=\psi(1)$.
The conclusion follows.
\end{proof}


\begin{remark}
It is interesting to provide an alternative proof of the first point of Theorem \ref{thmMoser}, in terms of a sub-Riemannian Laplacian, as in \cite{KhesinLee}, and which does not use the exact reachability result established in Theorem \ref{ball}.
The argument goes as follows.

In the proof above, we used an argument in two steps, the first of which being the classical Moser trick. But we could have tried to construct a time-dependent vector field $X(\cdot)\in L^2(0,1;\Gamma^{s}(M))$ such that $X(t)\in\Delta$ almost everywhere (this condition ensuring that the generated flow $\varphi(\cdot)\in H^1(0,1;\mathcal{D}^{s}(M))$ be horizontal) and such that $L_{X(t)}\mu(t) + \dot\mu(t)=0$ almost everywhere, for some appropriate path $\mu(\cdot)\in H^1(0,1;\mathrm{Vol}^{s-1}_\alpha(M))$ such that $\mu(0)=\mu_0$ and $\mu(1)=\mu_1$.

This can be done, by searching $X(t)$ in the form $X(t)=\nabla_{SR} F(t)$, where the so-called horizontal gradient for the (finite-dimensional) sub-Riemannian structure on $M$. Recall that the horizontal gradient $\nabla_{SR} F\in\Delta$ is defined by $g(\nabla_{SR}F,v)=dF.v$ for every $v\in\Delta$; if the vector fields $X_1,\ldots,X_r$ are locally orthonormal then $\nabla_{SR} F = \sum_{i=1}^r(L_{X_i}F)X_i$ (see \cite{MBOOK}).
Then, writing $d\mu(t)=f(t)\, dx_g$, the condition $L_{X(t)}\mu(t) = - \dot\mu(t)$ gives $L_{X(t)}d\mu(t) = - \dot f(t)\, dx_g$, and, denoting by $\mathrm{div}_\mu$ the divergence operator associated with the volume form $\mu$ (defined by $\mathrm{div}_\mu(X)\, d\mu = L_Xd\mu$ for any vector field $X$), we have to solve $\mathrm{div}_{\mu(t)}(X(t))=-\dot f(t)/f(t)$. Since we posit $X(t)=\nabla_{SR} F(t)$, this gives $\triangle_{\mu(t)}F(t)=-\dot f(t)/f(t)$, where $\triangle_{\mu(t)}$ is the \textit{sub-Riemannian Laplacian} associated with the volume form $\mu(t)$ and the metric $g$. It is well known that, under the bracket-generating assumption (also called \textit{H\"ormander assumption}), $-\triangle_{\mu(t)}$ is a subelliptic nonnegative selfadjoint operator with discrete spectrum $0=\lambda_1<\lambda_2<\cdots<\lambda_n<\cdots$ with $\lambda_n\rightarrow+\infty$.
In particular, since $\dot f(\cdot)/f(\cdot)\in L^2(0,1;H^{s-1}(M))$, it follows that there exists a solution $F(\cdot)\in L^2(0,1;H^{s-1}(M))$ (defined up to additive constant), and hence $X(\cdot)\in L^2(0,1;H^{s-2}(M))$ (at least).
Note that this reasoning gives, finally, a less precise result than in Theorem \ref{thmMoser} (where, anyway, it is not required to use an hypoelliptic Laplacian).
\end{remark}

\section{Conclusion}
In this paper, we have provided a framework in order to define and analyze a strong (infinite-dimensional) right-invariant sub-Riemannian structure on the group of diffeomorphisms of a (finite-dimensional) manifold. We have shown how certain results from the finite-dimensional case can be established in this new context (such as reachability properties), and we have also highlighted some important differences, one of them, of particular interest, being the occurence of what we have called elusive geodesics.
Such geodesics are due to a discrepancy between the manifold topology and the topology induced by the sub-Riemannian distance on the group of diffeomorphisms, the latter being finer (but it may not correspond to the topology of a manifold). Indeed, restricting the structure to a subgroup of more regular diffeomorphisms turns certain elusive geodesics into normal geodesics, by adding new covectors to be used as initial momenta. This raises the open question of whether one could find a set of covectors large enough to encapsulate all geodesics, so that there would be no elusive geodesic.

Another open problem is to prove that the converse inequality of \eqref{inegballbox} in Theorem \ref{ball} holds true as well. This might require the generalization of the concept of privileged coordinates to the infinite dimensional case.

Finally, we stress that, in the present paper, we have focused on strong sub-Riemannian structures. A lot of  interesting problems are open for weak sub-Riemannian geometries. Their study is harder because the Hamiltonian is not always well defined. This is a well-known problem in the study of weak Riemannian metrics (see \cite{BHM2} for example), and requires a case-by-case analysis. We hope that the framework that we have developed here can serve as a base in order to address new problems for weak sub-Riemannian structures, with many promising applications such as, using methods similar to those of \cite{EM}, the investigation of fluids with non-holonomic constraints.

\appendix
\section{Proof of Theorem \ref{ball}}\label{appendix_proofball}
We follow the method used in \cite{MBOOK} to prove the Chow-Rashevski theorem. Of course, since we are in infinite dimension, some new difficulties occur. The proof goes in 6 steps.

\paragraph{Step 1: Reduction to a neighborhood of $e=\mathrm{Id}_m$.}

\begin{lemma}
The reachable set $\mathcal{R}(e)$ is a subgroup of $\mathcal{D}^s(M)$.
\end{lemma}

\begin{proof}
Let $\varphi,\psi\in \mathcal{R}(e)$. Since $d_{SR}$ is right-invariant, we have
$d_{SR}(\varphi\circ\psi^{-1},e)=d_{SR}(\varphi,\psi)\leq d_{SR}(\varphi,e)+d_{SR}(\psi,e)<+\infty$, and therefore $\varphi\circ\psi^{-1}\in \mathcal{R}(e)$.
\end{proof}

Using that lemma, to prove that $\mathcal{R}(e)=\mathcal{D}^s(M)$, it suffices to prove that $\mathcal{R}(e)$ contains a neighborhood of $e$. 
Moreover, since $d_{SR}$ is right-invariant, to prove that the topologies coincide, it suffices to prove that any sub-Riemannian ball centered at $e$ contains a neighborhood of $e$ for the intrinsic manifold topology of $\mathcal{D}^s(M)$ (this will imply as well that $\mathcal{R}(e)$ contains a neighborhood of $e$).

\paragraph{Step 2: Smooth parametrization of the horizontal distribution.} 
Let us recall that the parametrization of $\mathcal{H}^s$, given by
$$
(\varphi,X)\in \mathcal{D}^s(M)\times \mathcal{H}_e\mapsto X\circ\varphi\in T_\varphi\mathcal{D}^s(M),
$$
is only continuous. 
Because of that, it is not possible to compute in a ``blind way" Lie brackets of horizontal vector fields on $\mathcal{D}^s(M)$ (recall that a vector field $\textbf{X}:\mathcal{D}^s(M)\rightarrow TD^s(M)$ is \emph{horizontal} if $\textbf{X}(\varphi)\in \mathcal{H}_\varphi$, for every $\varphi\in \mathcal{D}^s(M)$).

\begin{remark}
In order to avoid any confusion between vector fields on $M$ and vector fields on $\mathcal{D}^s(M)$, we will write vector fields on the infinite-dimensional manifold $\mathcal{D}^s(M)$ with bold letters.
\end{remark}

To overcome the problem of the continuous parametrization of $\mathcal{H}^s$, and in view of computing Lie brackets (see Step 4 further), we will rather use the mapping defined in the following lemma, which provides a smooth parametrization (inspired from control theory).

\begin{lemma}\label{smoothpar}
The mapping
$$
\begin{array}{rcl}
\mathcal{D}^s(M)\times H^s(M,\R^r) & \longrightarrow & T\mathcal{D}^s(M) \\
(\varphi,u^1,\dots,u^r) & \longmapsto & \displaystyle {\bf X}^u(\varphi)=\sum_{i=1}^ru^i X_i\circ\varphi
\end{array}
$$
is smooth, and its image is equal to $\mathcal{H}^s$. In particular, any such ${\bf X}^u$ is a smooth horizontal vector field on $\mathcal{D}^s(M)$.
\end{lemma}

\begin{proof}
The mapping is clearly smooth, because in the sum $\sum_{i=1}^ru^i X_i\circ\varphi$, only the terms $X_i\circ\varphi$ depend on $\varphi$, and these terms are smooth with respect to $\varphi$ since the vector fields $X_i$ are smooth. Moreover, writing that $u^i X_i\circ\varphi =  ( (u^i\circ\varphi^{-1}) X_i ) \circ \varphi$, and noting that $u^i\circ\varphi^{-1} \in H^s(M)$, it follows that $
\sum_{i=1}^ru^i X_i\circ\varphi   \in \mathcal{H}_e\circ\varphi=\mathcal{H}_\varphi$, and therefore that any $ {\bf X}^u$ is an horizontal vector field on $\mathcal{D}^s(M)$. Conversely, any $Y_\varphi\in \mathcal{H}_\varphi$ can be written as $Y_\varphi=\sum_{i=1}^r (u^i\circ\varphi) X_i\circ\varphi={\bf X}^{u\circ\varphi}(\varphi)$. Hence the image of the mapping is equal to $\mathcal{H}^s$.
\end{proof}

\paragraph{Step 3: Length of integral curves of smooth horizontal vector fields.} 
It follows from Lemma \ref{smoothpar} that every vector field ${\bf X}^u$ generates a unique local flow $(t,\varphi,u)\mapsto\Phi(t,u,\varphi)$ on $\mathcal{D}^s(M)$, which is smooth. Moreover, any integral curve $t\mapsto\varphi(t)=\Phi(t,u,\varphi(0))$ of this flow is a smooth horizontal curve for the right-invariant sub-Riemannian structure on $\mathcal{D}^s(M)$ induced by $\mathcal{H}_e$. 

\begin{lemma}\label{calcleng}
For every (small enough) open subset $\mathcal{U}$ of $\mathcal{D}^s(M)$, there exists $C>0$ such that
$$
\left\langle\sum_{i=1}^r (u_i\circ\varphi^{-1}) X_i,\sum_{i=1}^r (u_i\circ\varphi^{-1}) X_i\right\rangle\leq C\sum_{i=1}^r\Vert u_i\Vert_{H^s}^2,
$$
for every $u\in H^s(M,\R^r)$ and every $\varphi\in \mathcal{U}$.
Therefore the length of the curve $t\mapsto\Phi(t,u,\varphi)$ is bounded above by
$\displaystyle C \left( \sum_{i=1}^r\Vert u_i\Vert_{H^s}^2 \right)^{1/2}$.
\end{lemma}

\begin{proof}
The mapping from $\mathcal{D}^s(M)\times H^s(M,\R^r)$ to $H^s(M,\R^r)$, defined by $(\varphi,u)\mapsto u\circ\varphi^{-1}=R_{\varphi^{-1}}u$, is continuous. This implies (see \cite{O}) that the mapping $\varphi\mapsto R_{\varphi^{-1}}$, defined on $\mathcal{D}^s(M)$ with values in the space of continuous linear operators on $H^s(M,\R^r)$, is locally bounded (although it may fail to be continuous). Since the Hilbert norm on $\mathcal{H}_e$ is equivalent to the $H^s$ norm, the result follows.
\end{proof}

\paragraph{Step 4: Lie brackets of horizontal vector fields.} 
For every $i\in\{1,\dots,r\}$ and every $u\in H^s(M)$, we define ${\bf X}^u_i$ by ${\bf X}^u_i(\varphi)=uX_i(\varphi)$. In other words, we have ${\bf X}^u_i={\bf X}^{(0,\dots,0,u,0,\dots,0)}$.
The vector field $ {\bf X}_i^u$ is smooth, for every $u\in H^s(M)$. Therefore we can compute Lie brackets of such vector fields. 

First of all, for every $(\varphi,\delta\varphi)\in T\mathcal{D}^s(M)$, and every $i\in\{1,\dots,r\}$, we have
$\d {\bf X}_i^u(\varphi) . \delta\varphi=\partial_\varphi\left(u X_i\circ\varphi\right)(\varphi) . \delta \varphi=u^i (dX_i\circ\varphi) . \delta \varphi$.
Hence, for every $\varphi\in \mathcal{D}^s(M)$, for all elements $u$ and $v$ of $H^s(M)$, and for all indices $i$ and $j$ in $\{1,\dots,r\}$, we have
$$
\left[\textbf{X}_i^u,\textbf{X}_j^v\right](\varphi)
=
\d{\bf X}_j^v(\varphi).{\bf X}_i^u(\varphi)-\d{\bf X}_i^u(\varphi).{\bf X}_j^v(\varphi)
=
v (\d X_j\circ\varphi).(uX_i\circ\varphi)-u (\d X_i\circ\varphi).(v X_j\circ\varphi)  .
$$
But since, obviously, one has $\d X_i(\varphi(x)).(v(x)X_j(\varphi(x)))= v(x)\d X_i(\varphi(x)).X_j(\varphi(x))$ for every $x\in M$, we obtain
$$
\left[\textbf{X}_i^u,\textbf{X}_j^v\right](\varphi)=uv\left(\d X_j(\varphi(x)).X_i(\varphi(x))-\d X_i(\varphi(x)).X_j(\varphi(x))\right)
=uv[X_i,X_j]\circ\varphi
=uvX_{i,j}\circ\varphi.
$$
By induction, we get the following lemma (recall that the smooth vector field $X_I$ on $M$ is defined in \eqref{cro} by $X_I=\left[{X}_{i_1},[\dots,[{X}_{i_{j-1}},{X}_{i_j}]\dots\right]$).

\begin{lemma}
Let $j\in\N^*$, let $I=(i_1,\dots,i_j)\in\{1,\dots,r\}^j$, let $\varphi\in\mathcal{D}^s(M)$, and let $u^1,\dots,u^j$ be elements of $H^s(M)$. Then
$$
\left[\textbf{X}_{i_1}^{u_1},[\dots,[\textbf{X}_{i_{j-1}}^{u_{j-1}},\textbf{X}_{i_j}^{u_j}]\dots\right](\varphi)=u_{i_1}\dots u_{i_j}X_I\circ\varphi .$$
\end{lemma}

\paragraph{Step 5: Taylor expansions of commutators of horizontal flows, and corresponding length.}

For every $i\in\{1,\dots,r\}$ and every $u_i\in H^s(M)$, we denote by $(t,\varphi)\mapsto\Phi_i^{u_i}(t,\varphi)$ the flow of ${\bf X}_i^{u_i}$ on $\mathcal{D}^s(M)$. Moreover, for every $j\in \N^*$ and every $I=(i_1,\dots,i_j)\in\{1,\dots,r\}^j$, we define
$$
\Phi_I(t,{u_1,\dots,u_j}) = 
\Phi^{-u_j}_{i_j}(t)\circ\dots\circ\Phi^{-u_1}_{i_1}(t)\circ\Phi^{u_j}_{i_j}(t)\circ\dots\circ\Phi^{u_1}_{i_1}(t).
$$

\begin{remark}\label{boxing}
Lemma \ref{calcleng} implies that, for every $\varphi\in\mathcal{D}^s(M)$, there exists $C>0$ such that, for all elements $u_1,\dots,u_j$ of $H^s(M)$, of norm small enough, the diffeomorphisms $\varphi$ and $\Phi_I(1,{u_1,\dots,u_j})(\varphi)$ can be connected with a curve of length less than $2C(\Vert u_1\Vert_{H^s}+\dots+\Vert u_j\Vert_{H^s})$, hence
$$
d_{SR}(\varphi,\Phi_I(1,{u_1,\dots,u_j})(\varphi))\leq 2C(\Vert u_1\Vert_{H^s}+\dots+\Vert u_j\Vert_{H^s}).
$$
\end{remark}

The key lemma is the following.

\begin{lemma}\label{taylexp}
Let $\varphi\in \mathcal{D}^s(M)$, let $j\geq 2$ be an integer, and let $I=(i_1,\dots,i_j)\in\{1\dots,r\}^j$. Then
\begin{equation}\label{tayexp}
\Phi_I(1,{u_1,\dots,u_j})(\varphi)=\varphi+u_1\dots u_jX_I\circ\varphi+\mathrm{o}(\Vert (u_1,\dots,u_j)\Vert_{H^s}^{j}) .
\end{equation}
\end{lemma}

\begin{proof}
From the definition of a Lie bracket, it is well known that, for fixed $u_1,\dots,u_j$ in $H^s(M)$ and for small $t\in\R$, one has
$$
\Phi_I(t,{u_1,\dots,u_j})(\varphi)=\varphi+t^j\left[\textbf{X}_{i_1}^{u_1},[\dots,[\textbf{X}_{i_{j-1}}^{u_{j-1}},\textbf{X}_{i_j}^{u_j}]\dots\right](\varphi)+\mathrm{o}(t^{j})=\varphi+t^ju_1\dots u_jX_I\circ\varphi+\mathrm{o}(t^{j+1}).
$$
Obviously, we have $t\textbf{X}_{i_k}^{u_k}=\textbf{X}_{i_k}^{tu_k}$ for every $k$, and hence $\Phi_I(t,u_1,\dots,u_j)(\varphi)=\Phi_I(1,{tu_1,\dots,tu_j})(\varphi)$.
As a consequence, if $\Phi_I^{u_1,\dots,u_j}(1,\varphi)$ has a Taylor expansion in $(u_1,\dots,u_j)$, then this expansion is given by \eqref{tayexp}. Since the term at the left-hand side of \eqref{tayexp} is smooth in $(\varphi,u)$, it has a Taylor expansion in $u$ of order $j$. The result follows.
\end{proof}

\paragraph{Step 6: end of the proof.}
Since the family $(X_1,\dots,X_r)$ is bracket-generating, there exist subsets $I_1,\dots,I_m$ of $\{1,\dots,r\}$, of increasing cardinals $j_k=\vert I_k\vert$, such that any $X\in \Gamma^s(TM)$ is an a linear combination of $X_{I_1},\dots,X_{I_m}$ (with coefficients in $H^s(M)$). For every $k\in \{1,\dots,m\}$, we consider the mapping $\phi_k$ defined on a neighborhood of $(0,e)$ in $H^s(M)\times\mathcal{D}^s(M)$, with values in a neighborhood of $e$ in $\mathcal{D}^s(M)$, given by
$$
\phi_k(u,\varphi)=\Phi_{I_k}\left(1,\Vert u\Vert_{H^s}^{\frac{1-j_k}{j_k}}u,\Vert u\Vert_{H^s}^{1/j_k},\dots,\Vert u\Vert_{H^s}^{1/j_k}\right)(\varphi).
$$
This mapping is smooth with respect to $u$ outside of $0$, and it is of class $\mathcal{C}^1$ around $0$ since, according to Lemma \ref{taylexp}, we have $\phi_k(u,\varphi)=\varphi+uX_{I_k}\circ\varphi+\mathrm{o}(\Vert u\Vert_{H^s})$, for every $\varphi\in\mathcal{D}^s(M)$.
Moreover, according to Remark \ref{boxing}, for every $\varphi\in\mathcal{D}^s(M)$, there exists a constant $C>0$ such that 
\begin{equation}\label{trlength}
d_{SR}(\varphi,\phi_k(u,\varphi))\leq C\Vert u\Vert_{H^s}^{1/j_k}.
\end{equation}
as soon as the $H^s$ norm of $u$ is small enough.

Finally, we consider the mapping $\phi$ defined on a neighborhood of $0$ in $H^s(M,\R^m)$, with values in $\mathcal{D}^s(M)$, defined by
$$
\phi(u_1,\dots,u_m)=\phi_m(u_m)\circ\dots\circ\phi_1(u_1)(e).
$$
The mapping $\phi$ is of class $\mathcal{C}^1$ around $0$, and its differential at $0$ is given by
$$
\d\phi(0).(\delta u_1,\dots,\delta u_m)=\sum_{k=1}^m u_k X_{I_k}\in \Gamma^s(TM),
$$
for every $(\delta u_1,\dots,\delta u_m)\in H^s(M,\R^m)$. The mapping $\d\phi(0)$ is surjective, since the vector fields $X_{I_1},\dots,X_{I_n}$ generates $TM$ by assumption, and therefore $\phi$ is a local submersion at $0$. Therefore, the image by $\phi$ of any (small enough) neighborhood of $0$ in $H^s(M,\R^m)$ is a neighborhood of $e$ in $\mathcal{D}^s(M)$. Let $\varepsilon>0$ be small enough that $\phi$ is defined and smooth on the ball $B$ of center 0 and radius $\varepsilon^{j_m}$ in $H^s(M,\R^m)$. It follows that there exists a neighborhood $\mathcal{U}_\varepsilon$ of $e$ in $\mathcal{D}^s(M)$ such that $\mathcal{U}\subset \phi(B(0,\varepsilon^{j_m})))$.
Besides, for $\varepsilon>0$ small enough, \eqref{trlength} also implies that there exists $C>0$ such that 
$$
d_{SR}(e,\phi(u_1,\dots,u_m))\leq 2C\sum_{k=1}^m\Vert u_k\Vert_{H^s}^{1/j_n}\leq 2mC\varepsilon, 
$$
for every $(u_1,\dots,u_m)\in B(0,\varepsilon^{j_m})$.
Hence, any sub-Riemannian ball centered at $e$ contains a neighborhood $\mathcal{U}$ of $e$ for the intrinsic manifold topology of $\mathcal{D}^s(M)$. 
The theorem is proved.

\end{document}